\documentclass[11pt,amssymb]{amsart}
\usepackage{amsmath,amssymb,amsfonts,euscript,enumerate}
\usepackage{amsthm}
\usepackage{color,wrapfig}
\usepackage{verbatim}
\usepackage{ulem}
\usepackage{graphicx}
\usepackage{color} % load color package
\newcommand{\nc}{\newcommand}

\newtheorem{thm}{Theorem}[section]

\newtheorem{corollary}[thm]{Corollary}

\newtheorem{definition}{Definition}[section]

\newtheorem{lem}[thm]{Lemma}

\newcommand{\beq}{\begin{equation}}
\newcommand{\eeq}{\end{equation}}
\newcommand{\bcl}{\begin{center}}
\newcommand{\ecl}{\end{center}}

\newcommand{\La}{\triangle}
\newcommand{\osc}{\operatornamewithlimits{osc}}
\newcommand{\D}{\nabla}

\newcommand{\supp}{\operatorname{supp}}
\newcommand{\R}{{\mathbb R}}
\newcommand{\Z}{{\mathbb Z}}
\newcommand{\N}{{\mathbb N}}

\newcommand{\vp}{\varphi}

\nc \sms {\smallskip}
\nc{\fp}{\noindent}
\nc{\qd}{\qquad\qquad}
\title[H\"older regularity and Uniqueness of the degenerate Keller-Segel system]{H\"older regularity and Uniqueness theorem on weak solutions to the degenerate Keller-Segel system}

\author[Sunghoon Kim]{Sunghoon Kim}
\address{Sunghoon Kim:
Department of Mathematics, School of Natural Sciences, The Catholic University of Korea,
43 Jibong-ro, Wonmi-gu, Bucheon-si, Gyeonggi-do, 420-743, Republic of Korea}
\email{math.s.kim@catholic.ac.kr}

\author[Ki-Ahm Lee]{Ki-Ahm Lee}
\address{Ki-Ahm Lee:
 School of Mathematical Sciences, Seoul National University,   Seoul 151-747, Republic of Korea \& Center for Mathematical Challenges, Korea Institute for Advanced Study, Seoul,130-722,  Republic of Korea
}
\email{kiahm@math.snu.ac.kr}

\begin{document}

\maketitle
\begin{abstract}
In this paper, we present local H\"older estimates for the degenerate Keller-Segel system \eqref{eq-cases-aligned-main-problem-of-Keller-Segel-System} below in the range of $m>1$ and $q>1$ before a blow-up of solutions. To deal with difficulties caused by  the degeneracy of the operator, we find uniform estimates depending sup-norm of the density function and modified the energy estimates and intrinsic scales considered in Porous Medium Equation. As its application, the uniqueness of weak solution to \eqref{eq-cases-aligned-main-problem-of-Keller-Segel-System} is also showed in the class of  H\"older continuous functions by proving $L^1$-contraction in this class.
\end{abstract}

%%%%%%%%%%%%%%%%%%%%%%%%%%%%%%%%%%%%%%%%%%%%%%%%%%%%%%%%%%%%%%%%%%%%
\setcounter{equation}{0}
\setcounter{section}{0}

\section{Introduction}\label{sec-intro}
\setcounter{equation}{0}
\setcounter{thm}{0}

We will investigate the regularity theory of solutions to degenerate parabolic equations and derive the uniqueness of such solutions. More precisely, we consider the following degenerate Keller-Segel system:
\begin{equation}\label{eq-cases-aligned-main-problem-of-Keller-Segel-System}
\begin{cases}
\begin{aligned}
&u_t=\nabla\cdot\left(\nabla u^m-u^{q-1}\nabla v\right), \qquad x\in\R^n,\quad t>0,\\
&\delta v_t=\La v-\gamma v+u \qquad \qquad \qquad  x\in\R^n,\quad t>0\\
&u(x,0)=u_0(x) \qquad \qquad \qquad x\in\R^n, 
\end{aligned}
\end{cases}
\tag{{\bf $KS_m$}}
\end{equation}
where $n\geq 2$ and $m\geq 1$, $q\geq 2$, $\gamma>0$, $\delta=1$ or $\delta=0$. The initial data $(u_0,v_0)$ is a non-negative function and in $\left(L^{1}\cap L^{\infty}\right)\times \left(L^1\cap H^1\cap W^{1,\infty}\right)$ with $u_0^{m}\in H^{1}(\R^n)$. The standard system with $m=1$, $q=2$ and $\delta=1$ was introduced by Keller and Segel \cite{KS} in 1970. They presented a mathematical formulation which is modelling aggregation process of amoebae by the chemotaxis on the simplest possible assumptions consistent with the known facts. Although they suggested the general system,  these types of systems become the most common formulations which are describing a part of cellular slime molds with the chemotaxis. Usually $u(x,t)$ stands for the cell density, $v(x,t)$ refers to as the chemotaxis concentration at place $x\in\R^n$, time $t>0$. We call \eqref{eq-cases-aligned-main-problem-of-Keller-Segel-System} the parabolic-elliptic Keller-Segel system and parabolic-parabolic Keller-Segel one when $\delta=0$ and $\delta=1$, respectively. Mathematical modelling for the Keller-Segel system with porous medium type diffusion can be found in \cite{SRLC} and nonlinear diffusion has been suggested by Hillen and Painter, \cite{HP}. \\
\indent In the parabolic-elliptic Keller-Segel system, it was shown by Sugiyama and Kunii \cite{SK} that for every non negative data $(u_0,v_0)\in \left(L^1\cap L^{\infty}\right)\times \left(L^1\cap H^1\cap W^{1,\infty}\right)$ with $u_0^m\in H^{1}(\R^n)$, there exists at least one weak solution of \eqref{eq-cases-aligned-main-problem-of-Keller-Segel-System} on $[0,\infty)$ when $m>q-\frac{2}{n}$. In the case of $1\leq m\leq q-\frac{2}{n}$, the global solution has been guaranteed when the initial data has a sufficiently small $L^{\frac{n(q-m)}{2}}$-norm.\\
\indent In parabolic-parabolic Keller-Segel system, it was also shown by Sugiyama and Kunii \cite{SK}, Ishida and Yokota \cite{IY1, IY2} that, for the same initial data  as parabolic-elliptic system, there exists at least one weak solution of \eqref{eq-cases-aligned-main-problem-of-Keller-Segel-System} on $[0,\infty)$ provided $m\geq 1$ and $m>q-\frac{2}{n}$. In the case of $1\leq m\leq q-\frac{2}{n}$, the local and global existence of weak solutions has been established for large and small initial data, respectively. It is still in progress to find solution of \eqref{eq-cases-aligned-main-problem-of-Keller-Segel-System} under more general conditions.  \\
\indent Large number of literatures on the regularity theory for the weak solutions of degenerate parabolic equations can be found. We refer the readers to the papers \cite{KL1}, \cite{K} for the Schauder and H\"older estimates of the porous medium equation in a bounded domain, respectively and to the paper \cite{KL2} for the H\"older estimates of the non-local version of fast diffusion equation.  We also refer to the papers \cite{KyLe1}, \cite{KyLe2} for the regularity theory of fully nonlinear integro-differential operators and the paper \cite{KKL} for the Harnack inequality of nondivergent parabolic operators on Riemannian manifolds.\\
\indent The ultimate goal in this article is to establish the regularity theory of the weak solutions of \eqref{eq-cases-aligned-main-problem-of-Keller-Segel-System} under some regularity conditions on the initial data $(u_0,v_0)$ and $(u,v)$. Especially, we want to investigate the local H\"older estimates of $u^m$. As an application, we will also show that the weak solutions of \eqref{eq-cases-aligned-main-problem-of-Keller-Segel-System} are unique in the class of H\"older continuous functions.\\
\indent Uniqueness results have been achieved under some regularity conditions. Sugiyama and Yahagi \cite{SY} introduced the space $L^s(0,T;L^p(\R^n))$ with some $s$ and $p$ for the uniqueness and continuity of weak solution with respect to the initial data. Based on the $L^1$-contraction principle, uniqueness was shown in \cite{SY} when the weak solution $u$ as well as $\partial_tu$ and $\nabla u^{q-1}$ has the additional regularity in $L^s(0,T;L^p(\R^n))$.\\
\indent Since the term $u^{q-1}\nabla v$ of \eqref{eq-cases-aligned-main-problem-of-Keller-Segel-System} can be thought as the perturbation of $\La u^m$, it seems to be reasonable to investigate uniqueness of solution of \eqref{eq-cases-aligned-main-problem-of-Keller-Segel-System} under the setting similar to the one for porous medium equation. On that point, Kagei, Kawakami and Sugiyama \cite{KKS} showed that weak solutions of \eqref{eq-cases-aligned-main-problem-of-Keller-Segel-System} with $\delta=0$ exist uniquely in spaces of H\"older continuous function in $x$ and $t$. Their work was extended to the parabolic-parabolic system by Miura and Sugiyama \cite{MS}. In \cite{MS}, they adapt the duality method (the existence result for the adjoint equation yields the uniqueness of solutions to the original equation) coupled with the vanishing viscosity duality method to achieve the uniqueness result.\\
\indent As pointed out above, several terms, for example $u_t$ and $\nabla u^{q-1}$, should have some regularity to obtain the $L^1$-contraction principle. However, at this time, we assume only $u, v\in L^{\infty}$. Hence, we need more steps to improve the regularity conditions of $u$ and $v$ to get the uniqueness of solutions.  On the other hand, since H\"older continuity is optimal regularity for the weak solution of the porous medium equation, our regularity result will be enough to investigate uniqueness of the solution of \eqref{eq-cases-aligned-main-problem-of-Keller-Segel-System} with assistance from the $L^1$ contraction principle. \\
\indent Throughout the paper, we are going to consider the weak solution of \eqref{eq-cases-aligned-main-problem-of-Keller-Segel-System}. The definition of weak solution is given as follow.

\begin{definition}
Let $T>0$. A pair $(u,v)$ of non-negative functions defined on $\R^n\times(0,T)$ is called a weak solution to \eqref{eq-cases-aligned-main-problem-of-Keller-Segel-System} on $[0,T)$ if
\begin{enumerate}
\item $u\in L^{\infty}(0,T;L^p(\R^n))$ $(\forall p\in[1,\infty])$, $u^m\in L^2\left(0,T;H^1(\R^n)\right)$,
\item $v\in L^{\infty}\left(0,T;H^1(\R^n)\right)$,
\item $(u,v)$ satisfies \eqref{eq-cases-aligned-main-problem-of-Keller-Segel-System} in the distribution sense, i.e., for every $\vp\in C^{\infty}_0(\R^n\times[0,T))$,
\begin{equation*}
\int_0^T\int_{\R}\left(\nabla u^m\cdot\nabla\vp-u^{q-1}\nabla v\cdot\nabla\vp-u\vp_t\right)\,dxdt=\int_{\R}u_0(x)\vp(x,0)\,dx,
\end{equation*}
\begin{equation*}
\int_0^T\int_{\R}\left(\nabla v\cdot\nabla\vp+\gamma v\vp-u\vp-\delta v\vp_t\right)\,dxdt=\int_{\R}v_0(x)\vp(x,0)\,dx.
\end{equation*}
\end{enumerate}
In particular, if $T>0$ can be taken arbitrary, then $(u,v)$ is called a global weak solution to \eqref{eq-cases-aligned-main-problem-of-Keller-Segel-System}.
\end{definition}

The paper is divided into three parts: In Part 1 (Section \ref{sec-important-estimates}), we intorduce some estimates which will be an important key for the result. In Part 2 (Section \ref{sec-holder-estimates}) we study the H\"older regularity of solution to the degenerate equation
\begin{equation*}
u_t=\nabla\left(\nabla u^m-u^{q-1}\nabla v\right).
\end{equation*}
It is simple observation that H\"older estimate of solution will be strongly effected by the coefficients and extra term $\nabla v$. Hence, proper conditions need to be imposed to $q$ and $\nabla v$ for the result. Part 3 (Section \ref{sec-uniqueness}) is devoted to the proof of $L^1$-contraction of solution to the degenerated Keller-Segel system. Based on estimates of Section 2 and 3, we will discuss the uniqueness of solution.

\section{Preliminary estimates}\label{sec-important-estimates}
\setcounter{equation}{0}
\setcounter{thm}{0}

In this section, we introduce several estimates that will be used repeatedly in the proof of H\"older regularity and in the uniqueness of solution. We first define the localized weight function $\psi$, which is introduced in \cite{Su}.
\begin{lem}
We define the localized weight function $\psi$ by
\begin{equation*}
\psi(r)=\begin{cases}
               \begin{array}{ll}
                   1 & \mbox{for $0\leq r<1$},\\
                   1-2(r-1)^2 & \mbox{for $1\leq r<\frac{3}{2}$},\\
                   2(2-r)^2 & \mbox{for $\frac{3}{2}\leq r<2$},\\
                   0 & \mbox{for $r\geq 2$}      
               \end{array}
            \end{cases}
\end{equation*}
and define $\psi_l(x):=\psi\left(\frac{|x|}{l}\right)$ for $x\in\R^n$ and $l=1,2,\cdots,.$ Then, there exist positive constants $c_1$ and $c_2$ depending only on $n$ such that
\begin{equation*}
\left|\nabla\psi_l(x)\right|\leq \frac{c_1}{l}\left(\psi_{l}(x)\right)^{\frac{1}{2}}, \qquad \left|\La\psi_l(x)\right|\leq \frac{c_2}{l^2} \qquad \mbox{for $x\in\R^n$}.
\end{equation*}
\end{lem}
Next, we reviews the well-known representation from the elliptic theory. We refer the reader to the books \cite{St} by Stein and the paper \cite{AS} by Aronszajn and Smith.

\begin{lem}[cf. Lemma 3 in \cite{SY}]
Let $n\geq 2$, $1\leq p\leq\infty$ and $f\in L^p(\R^n)$ and consider the problem
\begin{equation}\label{eq-parabolic-elliptic-condition-for-regularity-of-v-1}
-\La z+\gamma z=f \qquad \forall x\in\R^n
\end{equation}
Then, the function $z(x)\in L^p(\R^n)$ given by
\begin{equation*}
z(x)=\int G(x-y)f(y)\,dy
\end{equation*}
is the strong solution of \eqref{eq-parabolic-elliptic-condition-for-regularity-of-v-1} in $\R^n$. Here, $G(x)$ is the Bessel potential which can be expressed as
\begin{equation*}
G(x)=\gamma^{\frac{N}{2}-1}\cdot a_n\cdot e^{-\sqrt{\gamma}|x|}\int_0^{\infty}e^{-\sqrt{\gamma}|x|s}\left(s+\frac{s^2}{2}\right)^{\frac{n-3}{2}}\,ds
\end{equation*}
with the constant $a_n$ given by
\begin{equation*}
a_n=\left(2\left(2\pi\right)^{\frac{N-1}{2}}\Gamma\left(\frac{N-1}{2}\right)\right)^{-1}
\end{equation*}
Hence, for $f\in L^{\infty}(\R^n)$, it holds that $z$, $\nabla z\in L^{\infty}(\R^n)$ with the estimates
\begin{equation*}
\left\|z\right\|_{L^{\infty}}\leq \left\|G\right\|_{L^1}\left\|f\right\|_{L^{\infty}}, \qquad \left\|\nabla z\right\|_{L^{\infty}}\leq \left\|\nabla G\right\|_{L^1}\left\|f\right\|_{L^{\infty}}.
\end{equation*}
For $f\in L^{1}(\R^n)$, it holds that $\nabla z\in L^{p}(\R^n)$ for all $1\leq p\leq \frac{n}{n-1}$ with the estimates
\begin{equation*}
\left\|\nabla z\right\|_{L^{p}}\leq \left\|\nabla G\right\|_{L^p}\left\|f\right\|_{L^{\infty}}.
\end{equation*}
\end{lem}
By the semigroup theory with $L^p-L^q$ estimates for the heat semigroup, we obtain the following fundamental estimates of solution to the Cauchy problem for inhomogeneous linear heat equations which will play an important role in establishing the a priori estimates of solution $v$ in \eqref{eq-cases-aligned-main-problem-of-Keller-Segel-System}.
\begin{lem}[cf. Lemma 5 in \cite{SK}]\label{lem-Lemma-5-in-SK}
Let $n\in\N$, $T>0$, $\delta=1$, $1\leq p\leq \infty$ and $z_0\in L^p(\R^n)$. If $f\in L^1(0,T);L^p(\R^n)$, then 
\begin{equation*}\label{eq-cases-aligned-inhomogeneous-linear-heat-eq-for-L-infty}
\begin{cases}
\begin{aligned}
&z_t=\La z-\gamma z+f, \qquad x\in\R^n,\,\,t>0,\\
&z(x,0)=z_0(x), \qquad x\in\R^n
\end{aligned}
\end{cases}
\end{equation*} 
has a unique mild solution $z\in C\left([0,T];L^p(\R^n)\right)$ given by
\begin{equation*}
z(x,t)=e^{-t}e^{t\La}z_0+\int_0^te^{-(t-s)}e^{(t-s)\La}f(x,s)\,ds, \qquad t\in[0,T]
\end{equation*}
where $(e^{t\La}f)(x,t)=\left(4\pi t\right)^{-\frac{N}{2}}\int_{\R^n}e^{-\frac{|x-y|^2}{4t}}f(y,t)\,dy$.\\
Moreover, let $1\leq p'\leq p\leq \infty$, $\frac{1}{p'}-\frac{1}{p}<\frac{1}{n}$ and suppose that $z$ is the solution of
\begin{equation*}
\begin{cases}
\begin{aligned}
&\delta z_t=\La z-\gamma z+f, \qquad x\in\R^n,\,\,t>0,\\
&z(x,0)=z_0(x), \qquad x\in\R^n
\end{aligned}
\end{cases}
\end{equation*} 
where $z_0\in W^{1,p}(\R^n)$. If $f\in L^{\infty}(0,\infty;L^{p'}(\R^n))$, then there exists a constant $C>0$ such that
\begin{equation*}
\left\|z(t)\right\|_{L^p(\R^n)}\leq \left\|z_0\right\|_{L^p(\R^n)}+C\Gamma(a)\sup_{0<s<t}\left\|f(s)\right\|_{L^{p'}(\R^n)}
\end{equation*}
and
\begin{equation}\label{eq-L-p-estimate-of-v-by-u-2}
\left\|\nabla z(t)\right\|_{L^p(\R^n)}\leq \left\|\nabla z_0\right\|_{L^p(\R^n)}+C\Gamma(\tilde{a})\sup_{0<s<t}\left\|f(s)\right\|_{L^{p'}(\R^n)}.
\end{equation}
for $t\in[0,\infty)$, where $C$ is a positive constant independent of $p$, $\Gamma(\cdot)$ is the gamma function, and $a=1-\left(\frac{1}{p'}-\frac{1}{p}\right)\cdot\frac{N}{2}$, $\tilde{a}=\frac{1}{2}-\left(\frac{1}{p'}-\frac{1}{p}\right)\cdot\frac{N}{2}$.\\
\indent In addition, let $\left|\nabla^iz_0\right|\in L^p(\R^n)$, and $f\in L^2(0,T;W^{i-1,p}(\R^n))$ for $i=1,2,3.$ Then, it holds that
\begin{equation*}
\left\|\nabla^{i}z(t)\right\|_{L^p(\R^n)}^2\leq \left\|\nabla^iz_0\right\|_{L^p(\R^n)}^2+2(p+N-2)\int_{0}^{t}\left\|\nabla^{i-1}f(s)\right\|_{L^{p}(\R^n)}^2\,ds\qquad \mbox{for $t\in[0,\infty)$}.
\end{equation*}
\end{lem}

\section{H\"older continuity}\label{sec-holder-estimates}
\setcounter{equation}{0}
\setcounter{thm}{0}

This section will be devoted to proof of H\"older regularity of \eqref{eq-cases-aligned-main-problem-of-Keller-Segel-System}. We start by stating our first result, Sobolev-type inequality.

\begin{lem}\label{eq-Sobolev-Inequality}
Let $\eta(x,t)$ be a cut-off function compactly supported in $B_r$ and let $u$ be a function defined in $\R^n\times(t_1,t_2)$ for any $t_2>t_1>0$. Then $u$ satisfies the following Sobolev inequalities:
\begin{equation}\label{eq-first-sobolev-inequality-1}
\left\|\eta u\right\|_{L^{\frac{2n}{n-2}}(\R^n)}\leq C\left\|\nabla (\eta u)\right\|_{L^2(\R^n)}
\end{equation}
and
\begin{equation}\label{eq-weighted-sobolev-inequality-for-holder}
\begin{aligned}
\|\eta\, u\|^2_{L^2(t_1,t_2;L^{2}(\R^n))}\leq C\Big(\sup_{t_1\leq t\leq t_2}\|\eta\, u\|^2_{L^{2}(\R^n)}+\|\D
(\eta\, u)\|^2_{L^2(t_1,t_2;L^2(\R^n))}\Big)\,
|\{\eta\, u>0\}|^{\frac{2}{n+2}}
\end{aligned}
\end{equation}
for some $C>0$.
\end{lem}
\begin{proof}
Since the first inequality is a well-known result, it suffices to prove the second inequality. We will use a modification of the technique of \cite{Zh} and \cite{KL2} to prove the lemma. Let $\chi_{\eta u}(x,t)$ be
the function with
\begin{equation*}
\chi_{_{\eta u}}=\begin{cases}
                     \begin{array}{ccc}
                           1 & \qquad &\eta u > 0\\
                           0 & \qquad &\eta u=0,
                     \end{array}
                \end{cases}
\end{equation*}
then we have
\begin{equation*}
\|\eta u\|^2_{L(t_1,t_2;L^2(\R^n)}=\int_{t_1}^{t_2}\int_{\R^n}|\eta
u|^2dxdt=\int_{t_1}^{t_2}\int_{\R^n}|\eta u|^2\chi_{_{\eta v}}dxdt.
\end{equation*}
Thus, by the H\"older inequality, we obtain
\begin{equation*}
\begin{aligned}
\|\eta u\|^2_{L(t_1,t_2;L^2(\R^n)}&\leq
\int_{t_1}^{t_2}\left(\int_{\R^n}|\eta
u|^{2\cdot\left(\frac{n+2}{n}\right)}dx\right)^{\frac{n}{n+2}}\left(\int_{\R^n}\left(\chi_{_{\eta
u}}\right)^{\frac{n+2}{2}}dx\right)^{\frac{2}{n+2}}dt\\
&\leq \left(\int_{t_1}^{t_2}\int_{\R^n}|\eta
u|^{2\cdot\left(\frac{n+2}{n}\right)}dxdt\right)^{\frac{n}{n+2}}|\{\eta
u>0\}|^{\frac{2}{n+2}}.
\end{aligned}
\end{equation*}
Now we use interpolation inequalities of $L^p$ spaces,
\begin{equation*}
\begin{aligned}
&\left(\int_{t_1}^{t_2}\int_{\R^n}|\eta
u|^{2\cdot\left(\frac{n+2}{n}\right)}dxdt\right)^{\frac{n}{n+2}}\leq
\left[\int_{t_1}^{t_2}\left(\int_{\R^n}|\eta
u|^2dx\right)^{(1-\beta)\overline{p}}\left(\int_{\R^n}|\eta u|^{2\overline{q}}dx\right)^{\frac{\beta \overline{p}}{\overline{q}}}dt\right]^{\frac{1}{\overline{p}}}
\end{aligned}
\end{equation*}
where $1<\overline{p}=\frac{n+2}{n}<\overline{q}$ and
$\frac{1}{\overline{p}}=\frac{\beta}{\overline{q}}+\frac{1-\beta}{1}$,
$\left(\beta=\frac{1}{\overline{p}}\right)$. Thus
\begin{equation*}
\begin{aligned}
&\left(\int_{t_1}^{t_2}\int_{\R^n}|\eta
u|^{2\cdot\left(\frac{n+2}{n}\right)}dxdt\right)^{\frac{n}{n+2}}\leq \sup_{t_1\leq t\leq
t_2}\left(\int_{\R^n}|\eta
u|^2dx\right)+\int_{t_1}^{t_2}\left(\int_{\R^n}|\eta
u|^{2\overline{q}}dx\right)^{\frac{1}{\overline{q}}}dt.
\end{aligned}
\end{equation*}
where $\overline{q}=\frac{n}{n-2}$. By \eqref{eq-first-sobolev-inequality-1}, we have
\begin{equation*}
\left(\int_{\R^n}|\eta u|^{2\overline{q}}dx\right)^{\frac{1}{\overline{q}}}\leq C\int_{\R^n}|\nabla(\eta u)|^{2}dx,
\end{equation*}
which gives the desired result.
\end{proof}

In order to develop the H\"older regularity method, it is necessary to localize the energy inequality by space and time truncation. Hence, we need to derive Local Energy Estimate in the interior of $\R^n\times(0,\infty)$ which will be the main tools in establishing H\"older estimates of solutions. Assume that $B_r$ is the ball of radius $r$ centered at $0\in\R^n$.
\begin{lem}\label{lem-Loc-Ene-Est-1}
Let $t_1<t_2$, $q>1$, $m>1$ and let $(u,v)$ be a weak solution pair to \eqref{eq-first-sobolev-inequality-1}. Then, there exists a constant $C<\infty$ such that for cut-off function $\eta$ compactly supported in $B_r$ and for every level $k$,
\begin{equation}\label{Holder-energy-estimate}
\begin{aligned}
&\int_{B_r\times
  \{t_2\}}\eta^2\left[\int_0^{(u^m-k)_\pm}(k\pm\xi)^{\frac{1}{m}-1}\xi\,
d\xi\right]\, dx+\int_{t_1}^{t_2}\int_{B_r}\left|\nabla (\eta (u^m-k)_{\pm})\right|^2dx\,dt\\
& \qquad  \leq Cm\Bigg(\int_{t_1}^{t_2}\int_{B_r} (u^m-k)_\pm^2\,\left|\nabla\eta\right|^2dx\,dt+\int_{t_1}^{t_2}\int_{B_r}
\left[\int_0^{(u^m-k)_\pm}(k\pm\xi)^{\frac{1}{m}-1}\xi\,
d\xi\right]|\eta\eta_t|\, dx\,dt\\
&\qquad \qquad +\int_{B_r\times\{t_1\}}
\eta^2\left[\int_0^{(u^m-k)_\pm}(k\pm\xi)^{\frac{1}{m}-1}\xi\,
d\xi\right]\, dx+\int_{t_1}^{t_2}\int_{B_r\cap\{(u^m-k)_\pm>0\}}u^{2(q-1)}\eta^2\left|\nabla v\right|^2\,dxdt\Bigg).
\end{aligned}
\end{equation}
\end{lem}
\begin{proof}
We will use a modification of the technique of \cite{KL2} to prove the lemma. We have for every $t_1<t<t_2$:
\begin{equation*}
\begin{aligned}
\int_{B_r}\left(u^m-k\right)_\pm\eta^2u_t\,dx+\int_{B_r}\nabla\left(\left(w^{m}-k\right)_\pm\eta^2\right)\cdot\nabla(u^m)\,\,dx=\int_{B_r}u^{q-1}\nabla\left(\left(u^m-k\right)_\pm\eta^2\right)\cdot\nabla v\,dx.
\end{aligned}
\end{equation*}
Then,
\begin{equation}\label{eq-aligned-inequalities-for-energy-inequality-step-1}
\begin{aligned}
&\frac{1}{m}\int_{B_r}\frac{d}{dt}\left[\int_{0}^{\left(u^m-k\right)_\pm}\left(k\pm\xi\right)^{\frac{1}{m}-1}\xi\,d\xi\right]\eta^2\,dx+\int_{B_r}\left|\nabla\left(\left(u^m-k\right)_\pm\eta\right)\right|^2\,dx\\
&\qquad \qquad \qquad \leq \int_{B_r}\left(u^m-k\right)_\pm\left|\nabla\eta\right|^2\,dx+\int_{B_r}u^{q-1}\eta\nabla\left(\left(u^m-k\right)_\pm\eta\right)\cdot\nabla v\,dx\\
& \qquad \qquad \qquad \qquad \qquad \qquad +\int_{B_r}u^{q-1}\left(u^m-k\right)_\pm\eta\nabla\eta\cdot\nabla v\,dx.
\end{aligned}
\end{equation}
By Young's inequality, 
\begin{equation}\label{eq-aligned-inequalities-for-energy-inequality-tool-1}
\begin{aligned}
&\int_{B_r}u^{q-1}\eta\nabla\left(\left(u^m-k\right)_\pm\eta\right)\cdot\nabla v\,dx\\
&\qquad \qquad \leq \frac{1}{2}\int_{B_r}\left|\nabla\left(\left(u^m-k\right)_\pm\eta\right)\right|^2\,dx+\frac{1}{2}\int_{B_r\cap\left\{(u^m-k)_\pm>0\right\}}u^{2(q-1)}\eta^2\left|\nabla v\right|^2\,dx
\end{aligned}
\end{equation}
and
\begin{equation}\label{eq-aligned-inequalities-for-energy-inequality-tool-2}
\begin{aligned}
&\int_{B_r}u^{q-1}\left(u^m-k\right)_\pm\eta\nabla\eta\cdot\nabla v\,dx\\
&\qquad \qquad \leq \frac{1}{2}\int_{B_r}\left(u^m-k\right)_\pm\left|\nabla\eta\right|^2\,dx+\frac{1}{2}\int_{B_r\cap\{(u^m-k)_\pm>0\}}u^{2(q-1)}\eta^2\left|\nabla v\right|^2\,dx.
\end{aligned}
\end{equation}
Putting \eqref{eq-aligned-inequalities-for-energy-inequality-tool-1} and \eqref{eq-aligned-inequalities-for-energy-inequality-tool-2} in \eqref{eq-aligned-inequalities-for-energy-inequality-step-1} and integrating it in over $(t_1,t_2)$, we have
\begin{equation}\label{eq-aligned-inequalities-for-energy-inequality-step-2}
\begin{aligned}
&\frac{1}{m}\int_{B_r\times
  \{t_2\}}\eta^2\left[\int_0^{(w-k)_\pm}(k\pm\xi)^{\frac{1}{m}-1}\xi\,
d\xi\right]\, dx+\frac{1}{2}\int_{t_1}^{t_2}\int_{B_r}\left|\nabla (\eta (u^m-k)_\pm)\right|^2dx\,dt\\
&\qquad \leq \frac{3}{2}\int_{t_1}^{t_2}\int_{B_r} (u^m-k)_\pm^2\,\left|\nabla\eta\right|^2dx\,dt+\frac{2}{m}\int_{t_1}^{t_2}\int_{B_r}
\left[\int_0^{(w-k)_\pm}(k\pm\xi)^{\frac{1}{m}-1}\xi\,
d\xi\right]|\eta\eta_t|\, dx\,dt\\
&\qquad \qquad +\frac{1}{m}\int_{B_r\times\{t_1\}}
\eta^2\left[\int_0^{(w-k)_\pm}(k\pm\xi)^{\frac{1}{m}-1}\xi\,
d\xi\right]\, dx+\int_{t_1}^{t_2}\int_{B_r\cap\{(u^m-k)_\pm>0\}}u^{2(q-1)}\eta^2\left|\nabla v\right|^2\,dxdt.
\end{aligned}
\end{equation}
and consequently obtain \eqref{Holder-energy-estimate}.
\end{proof}
\begin{lem}\label{lemma-control-the-last-term-by-A-l-r-measure=at=t}
Let $A^{\pm}_{l,r}(t)=\left\{x\in B_{r}:(u^m(x,t)-l)_\pm>0\right\}$ and let $\eta$ be a cut-off function compactly supported in $B_r$. Then, there exist constants $\tilde{q}$, $\tilde{r}$, $\kappa$ and  $I>0$ such that
\begin{equation*}
\int_{t_1}^{t_2}\int_{B_r\cap\{(u^m-k)_{\pm}>0\}}u^{2(q-1)}\eta^2\left|\nabla v\right|^2\,dx< I\left(\int_{t_1}^{t_2}\left|A^{\pm}_{k,r}(t)\right|^{\frac{\tilde{r}}{\tilde{q}}}\,dt\right)^{\frac{2}{\tilde{r}}\left(1+\kappa\right)}.
\end{equation*}
\end{lem}
\begin{proof}
We use a modification of the technique of \cite{Di} to prove the lemma. We first choose constants $a_1$, $a_2\geq 1$ satisfying 
\begin{equation}\label{eq-relation-between-a-1-and-a-2-for-smooth-computaiotion}
\frac{1}{a_2}+\frac{n}{2a_1}=1-\kappa_1
\end{equation}
for some constant $0<\kappa_1<1$. Then, by H\"older inequality, 
\begin{equation}\label{eq-aligned-split-the-last-term-of-Energy-Inequality-by-Holder-inequality}
\begin{aligned}
&\int_{t_1}^{t_2}\int_{B_r\cap\{(u^m-k)_\pm>0\}}u^{2(q-1)}\eta^2\left|\nabla v\right|^2\,dxd\tau\\
&\qquad \qquad \leq \int_{t_1}^{t_2}\left(\int_{B_r}u^{2a_1(q-1)}\eta^{2a_1}|\nabla v|^{2a_1}\,dx\right)^{\frac{1}{a_1}}\left|A^{\pm}_{k,r}(\tau)\right|^{1-\frac{1}{a_1}}\,d\tau\\
&\qquad \qquad \leq \left(\int_{t_1}^{t_2}\left(\int_{B_r}u^{2a_1(q-1)}\eta^{2a_1}|\nabla v|^{2a_1}\,dx\right)^{\frac{a_2}{a_1}}\,d\tau\right)^{\frac{1}{a_2}}\left(\int_{t_1}^{t_2}\left|A^{\pm}_{k,r}(\tau)\right|^{\left(1-\frac{1}{a_1}\right)\frac{a_2}{a_2-1}}\,d\tau\right)^{1-\frac{1}{a_2}}
\end{aligned}
\end{equation}
Since $u$ is a weak solution of \eqref{eq-cases-aligned-main-problem-of-Keller-Segel-System}, by Lemma \ref{lem-Lemma-5-in-SK}, there is a constant $I>0$ such that
\begin{equation*}
\left(\int_{t_1}^{t_2}\left(\int_{B_r}u^{2a_1(q-1)}\eta^{2a_1}|\nabla v|^{2a_1}\,dx\right)^{\frac{a_2}{a_1}}\,d\tau\right)^{\frac{1}{a_2}}\leq I.
\end{equation*}
Combining this with \eqref{eq-aligned-split-the-last-term-of-Energy-Inequality-by-Holder-inequality},
\begin{equation}\label{eq-applying-regularity-of-u-and-v-for-simplification}
\int_{t_1}^{t_2}\int_{B_r\cap\{(u^m-k)_\pm>0\}}u^{2(q-1)}\eta^2\left|\nabla v\right|^2\,dxd\tau\leq I\left(\int_{t_1}^{t_2}\left|A^{\pm}_{k,r}(\tau)\right|^{\left(1-\frac{1}{a_1}\right)\frac{a_2}{a_2-1}}\,d\tau\right)^{1-\frac{1}{a_2}}
\end{equation}
For simplification, we let
\begin{equation}\label{eq-definition-of-constant-tilde-q-and-kappa}
\tilde{q}=\frac{2\tilde{a_1}(1+\kappa)}{a_1-1}, \qquad  \tilde{r}=\frac{2a_2(1+\kappa)}{a_2-1}  \qquad \mbox{and} \qquad \kappa=\frac{2}{n}\kappa_1.
\end{equation}
Then, by \eqref{eq-applying-regularity-of-u-and-v-for-simplification} and \eqref{eq-definition-of-constant-tilde-q-and-kappa},
\begin{equation*}
\int_{t_1}^{t_2}\int_{B_r\cap\{(u^m-k)_\pm>0\}}u^{2(q-1)}\eta^2\left|\nabla v\right|^2\,dx\leq I\left(\int_{t_1}^{t_2}\left|A^{\pm}_{k,r}(\tau)\right|^{\frac{\tilde{r}}{\tilde{q}}}\,d\tau\right)^{\frac{2}{\tilde{r}}\left(1+\kappa\right)}
\end{equation*}
and the lemma follows.
\end{proof}
By Lemma \ref{lem-Loc-Ene-Est-1} and Lemma \ref{lemma-control-the-last-term-by-A-l-r-measure=at=t}, we have the following Local Energy Estimate.
\begin{corollary}[Local Energy Estimate]\label{cor-Loc-energy-estimate-with-constant-L03}
Let $t_1<t_2$, $q>1$, $m>1$ and let $(u,v)$ be a weak solution pair to \eqref{eq-first-sobolev-inequality-1}. We also let $\tilde{p}$, $\tilde{q}$, $\tilde{r}$, $\kappa$, $I>0$ be given by Lemma \ref{lemma-control-the-last-term-by-A-l-r-measure=at=t}. Then, there exists a constant $C<\infty$ such that for cut-off function $\eta$ compactly supported in $B_r$ and for every level $k$,
\begin{equation}\label{Holder-energy-estimate}
\begin{aligned}
&\int_{B_r\times
  \{t_2\}}\eta^2\left[\int_0^{(u^m-k)_\pm}(k\pm\xi)^{\frac{1}{m}-1}\xi\,
d\xi\right]\, dx+\int_{t_1}^{t_2}\int_{B_r}\left|\nabla (\eta (u^m-k)_\pm)\right|^2dx\,dt\\
&\qquad \leq Cm\Bigg(\int_{t_1}^{t_2}\int_{B_r} (u^m-k)_\pm^2\,\left|\nabla\eta\right|^2dx\,dt+\int_{t_1}^{t_2}\int_{B_r}\left[\int_0^{(u^m-k)_\pm}(k\pm\xi)^{\frac{1}{m}-1}\xi\,d\xi\right]|\eta\eta_t|\, dx\,dt\\
&\qquad \qquad +\int_{B_r\times\{t_1\}}
\eta^2\left[\int_0^{(u^m-k)_\pm}(k\pm\xi)^{\frac{1}{m}-1}\xi\,
d\xi\right]\, dx+\left(\int_{t_1}^{t_2}\left|A^{\pm}_{k,r}(\tau)\right|^{\frac{\tilde{r}}{\tilde{q}}}\,d\tau\right)^{\frac{2}{\tilde{r}}\left(1+\kappa\right)}\Bigg).
\end{aligned}
\end{equation}
where $A^{\pm}_{l,r}(t)=\left\{x\in B_{r}:(u^m(x,t)-l)_\pm>0\right\}$.
\end{corollary}

From now on, we start the story of H\"older continuity of the solution $u$ of \ref{eq-cases-aligned-main-problem-of-Keller-Segel-System}. To develop the H\"older regularity method, we need to handle the difficulty from the degeneracy. To get over it, we use the technique developed in  \cite{CD}, \cite{Di}, \cite{DK}, \cite{KL2}.\\
\indent The key idea of the proof is to work with cylinders whose dimensions are suitably rescaled to reflect the degeneracy exhibited by the equation. To make this precise, fix $(x_0,t_0)\in\R^n\times(0,T]$, for some $T>0$, and construct the cylinder
\begin{equation*}
\left[(x_0,t_0)+Q(2R,R^{2-\epsilon})\right]\subset\R^n\times(0,T], \qquad \left(0<R\leq 1\right)
\end{equation*}
where $\epsilon$ is a small positive number to be determined later. After a translation we may assume that $(x_0,t_0)=(0,0)$. Set
\begin{equation*}
\mu^+=\sup_{Q(2R,R^{2-\epsilon})}u^m, \qquad \mu^-=\inf_{Q(2R,R^{2-\epsilon})}u^m, \qquad \omega=\osc_{Q(2R,R^{2-\epsilon})} u^m=\mu^+-\mu^-.
\end{equation*}
and construct the cylinder
\begin{equation}\label{eq-construction-of-cylinder-with-some-constant-A-which-is-bigger}
Q\left(R,a_0^{-\alpha}R^2\right)=B_{R}\times\left(-a_0^{-\alpha}R^2,0\right), \qquad \left(a_0=\frac{\omega}{A}, \,\,\alpha=1-\frac{1}{m}\right)
\end{equation}
where $A$ is a constant to be determined later only in terms of the data. We will assume that
\begin{equation}\label{eq-condition-between-R-and-theta-alpha--1}
a_0^{\alpha}=\left(\frac{\omega}{A}\right)^{\alpha}>R^{\epsilon}.
\end{equation}
By \eqref{eq-construction-of-cylinder-with-some-constant-A-which-is-bigger} and \eqref{eq-condition-between-R-and-theta-alpha--1}, it can be easily checked that
\begin{equation*}
Q\left(R,a_0^{-\alpha}R^2\right)\subset Q(2R,R^{2-\epsilon}) \qquad \mbox{and} \qquad \osc_{Q\left(R,a_0^{-\alpha}R^2\right)}u^m\leq \omega.
\end{equation*}
\indent To begin the proof for the H\"older estimates, we consider sub-cylinders of smaller size in $Q(R,a_0^{-\alpha}R^2)$ constructed as follows. For any integer $s_0>0$, let $s_0$ be the smallest integer such that
\begin{equation}\label{eq-condition-for-s-0}
\frac{\omega}{2^{s_0}}<1
\end{equation}
and construct cylinders
\begin{equation*}
Q(R,\theta_0^{-\alpha}R^2), \qquad \left(\theta_0=\frac{\omega}{2^{s_0}}\right).
\end{equation*}
If the number $A$ is chosen larger than $2^{s_0}$, These are contained inside $Q(R,a_0^{-\alpha}R^2)$ and, by \eqref{eq-condition-between-R-and-theta-alpha--1},
\begin{equation}\label{eq-condition-between-R-and-theta-alpha}
\theta_0^{\alpha}=\left(\frac{\omega}{\theta_0}\right)^{\alpha}>\left(\frac{A}{2^{s_0}}\right)^{\alpha}R^{\epsilon}>R^{\epsilon}.
\end{equation}
We now first state the first alternative in this section states:
\begin{lem}\label{lem-first-alternative-for-Holder-estimate}
There exist positive numbers $\rho$ independent of $R$, $A$, $\mu^{\pm}$ and $\omega$ such that if 
\begin{equation}\label{eq-condition-1-of-first-alternatives-for-Holder}
\left|\left\{(x,t)\in Q\left(R,\theta_0^{-\alpha}R^2\right):\,\, u^m(x,t)<\mu^-+\frac{\omega}{2^{s_0}}\right\}\right|<\rho\left|Q\left(R,\theta_0^{-\alpha}R^2\right)\right|,
\end{equation}
then,
\begin{equation*}
u(x,t)>\mu^-+\frac{\omega}{2^{s_0+1}}, \qquad \mbox{for all $(x,t)\in Q\left(\frac{R}{2},\theta_0^{-\alpha}\left(\frac{R}{2}\right)^2\right)$}.
\end{equation*}
\end{lem}

\begin{proof}
We will use a modification of the technique of \cite{Di} to prove the lemma. Set, for any non-negative integer $i$,
\begin{equation*}
R_i=\frac{R}{2}+\frac{R}{2^{i+1}} \qquad \mbox{and} \qquad l_i=\mu^-+\left(\frac{\omega}{2^{s_0+1}}+\frac{\omega}{2^{i+s_0+1}}\right).
\end{equation*}
We also denote
\begin{equation*}
A(l_i,R_i)=\left\{(x,t)\in Q\left(R,\theta_0^{-\alpha}R^2\right):u^m(x,t)<l_i\right\}.
\end{equation*}
From the definition, we have
\begin{equation*}
\left|A(l_i,R_i)\right|=\int_{-\theta_0^{-\alpha}R^{2}_i}^0\left|\left\{x\in B_{R_i}:u^m(x,t)<l_i\right\}\right|\,dt.
\end{equation*}
We consider a cut-off function $\eta_i(x,t)$ such that
\begin{equation*}
\begin{cases}
\begin{aligned}
&0\leq \eta_i\leq 1 \qquad \mbox{in $Q\left(R_i,\theta_0^{-\alpha}R_i^2\right)$},\\
&\eta_i=1 \qquad \mbox{in $Q\left(R_{i+1},\theta_0^{-\alpha}R_{i+1}^2\right)$},\\
&\eta_i=0 \qquad \mbox{on the parabolic boundary of $Q\left(R_i,\theta_0^{-\alpha}R_i^2\right)$},\\
&\left|\nabla\eta_i\right|\leq \frac{2^{i+2}}{R}, \qquad \left(\eta_i\right)_t\leq \frac{2^{2(i+2)}\theta_0^{\alpha}}{R^2}.
\end{aligned}
\end{cases}
\end{equation*}
We will use the Corollary \ref{cor-Loc-energy-estimate-with-constant-L03} for the function $u^{\ast}_i=\left(u^m-l_i\right)_-$ over the cylinder $Q\left(R_i,\theta_0^{-\alpha}R_i^2\right)$ where $i=0,1,2,\cdots$.  To control the first term in \eqref{Holder-energy-estimate}, we consider the function $F(\xi):\R^+\to\R$ defined by
\begin{equation*}
F(\xi)=\left(k-\xi\right)^{\frac{1}{m}-1}\xi,\qquad\left(k>0, \quad 0\leq \xi\leq \left(u^m-k\right)_-<k\right).
\end{equation*}
By simple computation, we have
\begin{equation}\label{eq-aligned-differentiability-of-F}
\begin{aligned}
F(\xi)=k(k-\xi)^{-\alpha}-(k-\xi)^{1-\alpha} \qquad \mbox{and} \qquad 0<u^m\leq k.
\end{aligned}
\end{equation}
Then, by \eqref{eq-aligned-differentiability-of-F},
\begin{equation}\label{eq-aligned-lower-bound-of-integral-for-F-with-first-assumption}
\begin{aligned}
\int_{0}^{\left(u^m-k\right)_-}F(\xi)\,d\xi&\geq k^{1-\alpha}(u^m-k)_--\frac{m}{m+1}\left(k^{2-\alpha}-\left(u^m\right)^{2-\alpha}\right)
\end{aligned}
\end{equation}
Since $2-\alpha>1-\alpha>0$, by \eqref{eq-aligned-lower-bound-of-integral-for-F-with-first-assumption}, we get 
\begin{equation}\label{eq-aligned-lower-bound-of-integral-for-F-with-second-assumption}
\begin{aligned}
\int_{0}^{\left(u^m-k\right)_-}F(\xi)\,d\xi&\geq (u^m-k)_-^{2-\alpha}-\frac{m}{m+1}(u^m-k)_-^{2-\alpha}=\frac{1}{m+1}(u^m-k)_-^{2-\alpha}
\end{aligned}
\end{equation}
By \eqref{eq-aligned-lower-bound-of-integral-for-F-with-second-assumption}, 
\begin{equation*}
\int_{B_r\times\{t_2\}}\eta^2\left[\int_0^{(u^m-k)_-}(k-\xi)^{\frac{1}{m}-1}\xi\,d\xi\right]\, dx\geq \frac{1}{m+1}\int_{B_r\times\{t_2\}}(u^m-k)_-^{-\alpha}\left[\eta\left(u^m-k\right)_-\right]^{2}\,dx.
\end{equation*}
Applying Corollary \ref{cor-Loc-energy-estimate-with-constant-L03} on $A(R_i,l_i)$ and multiplying by $\theta_0^{\alpha}$, we can get
\begin{equation}\label{eq-result-from-applied-by-energy-estimate}
\begin{aligned}
&\sup_{-\theta_0^{-\alpha}R_i^2<t<0}\left\|\eta_i u^{\ast}_i\right\|^2_{L^2(B_{R_i})}+\theta_0^{\alpha}\left\|\nabla\left(\eta_i u^{\ast}_i\right)\right\|^2_{L^2(Q\left(R_i,\theta_0^{-\alpha}R_i^2\right))}\\
& \qquad \leq Cm\theta_0^{\alpha}\left[\frac{2^{2(i+2)}\omega^2}{2^{2s_0}R_i^2}\left(1+2m\right)\int_{-\theta_0^{-\alpha}R_i^2}^0\left|A^{-}_{l_i,R_i}(t)\right|\,dt+ \left(\int_{-\theta_0^{-\alpha}R_i^2}^{0}\left|A^{-}_{l_i,R_i}(t)\right|^{\frac{\tilde{r}}{\tilde{q}}}\,dt\right)^{\frac{2}{\tilde{r}}\left(1+\kappa\right)}\right]
\end{aligned}
\end{equation}
for some constant $C>0$ where $A^{-}_{l,R}(t)=\left\{x\in B_R:u^m<l\right\}$. We introduce the change of time-variable $z=\theta_0^{\alpha}t$ which transforms $Q(R_i,\theta_0^{-\alpha}R_i^2)$ into
\begin{equation*}
Q_i=Q(R_i,R_i^2).
\end{equation*}
Setting also
\begin{equation*}
v(\cdot,z)=u\left(\cdot,\theta_0^{-\alpha}z\right) \qquad \mbox{and} \qquad \hat{\eta}_i(\cdot,z)=\eta_i(\cdot,\theta_0^{-\alpha}z),
\end{equation*}
the inequality \eqref{eq-result-from-applied-by-energy-estimate} can be written as
\begin{equation}\label{eq-result-from-applied-by-energy-estimate-after-change-of-variable}
\begin{aligned}
&\sup_{-R_i^2<t<0}\left\|\eta_i v^{\ast}_i\right\|^2_{L^2(B_{R_i})}+\left\|\nabla\left(\eta_i v^{\ast}_i\right)\right\|^2_{L^2(Q\left(R_i,-R_i^2\right))}\\
& \qquad \leq Cm\left[\frac{2^{2(i+2)}\omega^2}{2^{2s_0}R_i^2}\left(1+2m\right)Z_i+ \theta_0^{\frac{\alpha}{a_2}}\left(\int_{-R_i^2}^{0}\left|Z_i(z)\right|^{\frac{\tilde{r}}{\tilde{q}}}\,dt\right)^{\frac{2}{\tilde{r}}\left(1+\kappa\right)}\right]
\end{aligned}
\end{equation}
where
\begin{equation*}
Z_i(z)=\left\{x\in B_{R_i}:v(x,z)<l_i\right\} \qquad \mbox{and} \qquad Z_i=\int_{-R_i^2}^0\left|A_i(z)\right|\,dz.
\end{equation*}

By Lemma \ref{eq-Sobolev-Inequality} and \eqref{eq-result-from-applied-by-energy-estimate-after-change-of-variable}, we get
\begin{equation}\label{eq-aligned-first-of-two-inequalities-for-finding-relation-between-Z-is}
\begin{aligned}
&\left\|\eta_i v^{\ast}_i\right\|_{L^2(Q\left(R_i,R_i^2\right))}^2\\
& \qquad \leq Cm\left[\frac{2^{2(i+2)}\omega^2}{2^{2s_0}R_i^2}\left(1+2m\right)Z_i+ \theta_0^{\frac{\alpha}{a_2}}\left(\int_{-R_i^2}^{0}\left|Z_i(z)\right|^{\frac{\tilde{r}}{\tilde{q}}}\,dt\right)^{\frac{2}{\tilde{r}}\left(1+\kappa\right)}\right]Z_i^{\frac{2}{n+2}}\\
&\qquad \qquad =\frac{Cm2^{2(i+2)}\omega^2}{2^{2s_0}R_i^2}\left(1+2m\right)Z_i^{1+\frac{2}{n+2}}+Cm \theta_0^{\frac{\alpha}{a_2}}Z_i^{\frac{2}{n+2}}\left(\int_{-R_i^2}^{0}\left|Z_i(z)\right|^{\frac{\tilde{r}}{\tilde{q}}}\,dt\right)^{\frac{2}{\tilde{r}}\left(1+\kappa\right)}.
\end{aligned}
\end{equation}
We also have
\begin{equation}\label{eq-aligned-second-of-two-inequalities-for-finding-relation-between-Z-is}
\begin{aligned}
\int_{Q\left(R_i,R_i^2\right)}\left|\eta_iv^{\ast}_i\right|^2\,dxdt&\geq \left(l_{i+1}-l_i\right)^2\int_{-R_i^2}^0\left|\left\{(x,t)\in B_{R_{i+1}}:v^m<l_{i+1}\right\}\right|\,dt\\
&=\left(\frac{\omega}{2^{i+s_0+2}}\right)^2Z_{i+1}.
\end{aligned}
\end{equation}
By \eqref{eq-aligned-first-of-two-inequalities-for-finding-relation-between-Z-is} and \eqref{eq-aligned-second-of-two-inequalities-for-finding-relation-between-Z-is},
\begin{equation}\label{eq-recursion-relation-between-Z-n-s-1}
Z_{i+1}\leq \frac{Cm2^{4i+8}}{R_i^2}\left(1+2m\right)Z_i^{1+\frac{2}{n+2}}+Cm2^{2i+4}\theta_0^{-\alpha\left(\frac{2m}{m-1}-\frac{1}{a_2}\right)}Z_i^{\frac{2}{n+2}}\left(\int_{-R_i^2}^{0}\left|Z_i(z)\right|^{\frac{\tilde{r}}{\tilde{q}}}\,d\tau\right)^{\frac{2}{\tilde{r}}\left(1+\kappa\right)}.
\end{equation}
Divide by $\left|Q\left(R_{i+1},R_{i+1}^2\right)\right|$ and set the quantity
\begin{equation*}
X_i=\frac{Z_{i}}{\left|Q\left(R_i,R_i^2\right)\right|} \qquad \mbox{and} \qquad Y_i=\frac{1}{\left|B_{R_i}\right|}\left(\int_{-R_i^2}^{0}\left|Z_i(z)\right|^{\frac{\tilde{r}}{\tilde{q}}}\,d\tau\right)^{\frac{2}{\tilde{r}}}.
\end{equation*}
Then, we obtain from \eqref{eq-recursion-relation-between-Z-n-s-1} that
\begin{equation}\label{eq-first-for-X-i-Y-i}
X_{i+1}\leq C16^i\left(X_1^{1+\frac{2}{n+2}}+\theta_0^{-\alpha\left(\frac{2m}{m-1}-\frac{1}{a_2}\right)}R^{n\kappa}X_{i}^{\frac{2}{n+2}}Y_i^{1+\kappa}\right).
\end{equation}
If we choose $\epsilon$ small enough that
\begin{equation*}
\epsilon<n\kappa\left(\frac{2m}{m-1}-\frac{1}{a_2}\right)^{-1},
\end{equation*}
then by \eqref{eq-condition-between-R-and-theta-alpha} and \eqref{eq-first-for-X-i-Y-i}, we get 
\begin{equation}\label{eq-the-first-iteration-inequality-of-X-i-and-Y-i}
X_{i+1}\leq C{16}^i\left(X_i^{1+\frac{2}{n+2}}+X_{i}^{\frac{2}{n+2}}Y_i^{1+\kappa}\right) \qquad \forall n\in\Z^+
\end{equation}
where the constant $C$ depends on $n$, $m$ and $\kappa$. By an argument similar to \eqref{eq-aligned-second-of-two-inequalities-for-finding-relation-between-Z-is}, we have 
\begin{equation}\label{eq-aligned-computation-for-finding-another-iteration-of-Y-i-and-X-i-1}
\begin{aligned}
Y_{i+1}\left(l_i-l_{i+1}\right)^2&\leq \frac{1}{\left|B_{R_{i+1}}\right|}\left[\int_{-R_{i+1}^2}^{0}\left(\int_{B_{R_{i+1}}}\left(\eta_iv^{\ast}_i\right)^{\tilde{q}}\,dx\right)^{\frac{\tilde{r}}{\tilde{q}}}\,d\tau\right]^{\frac{2}{\tilde{r}}}\\
&\leq \frac{1}{\left|B_{R_{i+1}}\right|}\left[\int_{-R_i^2}^{0}\left(\int_{B_{R_{i}}}\left(\eta_iv^{\ast}_i\right)^{\tilde{q}}\,dx\right)^{\frac{\tilde{r}}{\tilde{q}}}\,d\tau\right]^{\frac{2}{\tilde{r}}}.
\end{aligned}
\end{equation}
Observe that, by \eqref{eq-relation-between-a-1-and-a-2-for-smooth-computaiotion} and \eqref{eq-definition-of-constant-tilde-q-and-kappa},
\begin{equation}\label{eq-relations-between-constnasts-after-a-little-changese3}
\frac{n}{\tilde{q}}+\frac{2}{\tilde{r}}=\frac{n}{2}.
\end{equation}
Then, by H\"older inequality, 
\begin{equation}\label{eq-aligned-computation-for-finding-another-iteration-of-Y-i-and-X-i-2}
\begin{aligned}
\left(\int_{B_{R_{i}}}\left(\eta_iv^{\ast}_i\right)^{\tilde{q}}\,dx\right)^{\frac{\tilde{r}}{\tilde{q}}}&=\left(\int_{B_{R_{i}}}\left(\eta_iv^{\ast}_i\right)^{\left\{n\left(\frac{\tilde{q}}{2}-1\right)\right\}+\left\{\tilde{q}-n\left(\frac{\tilde{q}}{2}-1\right)\right\}}\,dx\right)^{\frac{\tilde{r}}{\tilde{q}}}\\
&\leq \left(\int_{B_{R_{i}}}\left(\eta_iv^{\ast}_i\right)^{\frac{2n}{n-2}}\,dx\right)^{n\left(\frac{\tilde{q}}{2}-1\right)\cdot\frac{n-2}{2n}\cdot\frac{\tilde{r}}{\tilde{q}}}\cdot\left(\int_{B_{R_{i}}}\left(\eta_iv^{\ast}_i\right)^{2}\,dx\right)^{\frac{1}{2}\left(\tilde{q}-n\left(\frac{\tilde{q}}{2}-1\right)\right)\cdot\frac{\tilde{r}}{\tilde{q}}}\\
&=\left(\int_{B_{R_{i}}}\left(\eta_iv^{\ast}_i\right)^{\frac{2n}{n-2}}\,dx\right)^{\frac{n-2}{n}}\cdot\left(\int_{B_{R_{i}}}\left(\eta_iv^{\ast}_i\right)^{2}\,dx\right)^{\left(1-\frac{2}{\tilde{r}}\right)\cdot\frac{\tilde{r}}{2}}.
\end{aligned}
\end{equation}
By \eqref{eq-aligned-computation-for-finding-another-iteration-of-Y-i-and-X-i-1}, \eqref{eq-relations-between-constnasts-after-a-little-changese3} and \eqref{eq-aligned-computation-for-finding-another-iteration-of-Y-i-and-X-i-2},
\begin{equation*}
\begin{aligned}
Y_{i+1}\left(l_i-l_{i+1}\right)^2 &\leq \frac{1}{\left|B_{R_{i+1}}\right|}\left(\sup_{-R_i^2<t<0}\left\|\eta_i v^{\ast}_i\right\|^2_{L^2(B_{R_i})}\right)^{1-\frac{2}{\tilde{r}}}\left(\int_{-R_i^2}^0\left(\int_{B_{R_{i}}}\left(\eta_iv^{\ast}_i\right)^{\frac{2n}{n-2}}\,dx\right)^{\frac{n-2}{n}}\,d\tau\right)^{\frac{2}{\overline{r}}}\\
&\leq \frac{1}{\left|B_{R_{i+1}}\right|}\left(\sup_{-R_i^2<t<0}\left\|\eta_i v^{\ast}_i\right\|^2_{L^2(B_{R_i})}+\int_{-R_i^2}^0\left(\int_{B_{R_{i}}}\left(\eta_iv^{\ast}_i\right)^{\frac{2n}{n-2}}\,dx\right)^{\frac{n-2}{n}}\,d\tau\right).
\end{aligned}
\end{equation*}
Thus, by Sobolev inequality, \eqref{eq-first-sobolev-inequality-1},
\begin{equation*}
Y_{i+1}\left(l_i-l_{i+1}\right)^2 \leq \frac{1}{\left|B_{R_{i+1}}\right|}\left(\sup_{-R_i^2<t<0}\left\|\eta_i v^{\ast}_i\right\|^2_{L^2(B_{R_i})}+\left\|\nabla\left(\eta_i v^{\ast}_i\right)\right\|^2_{L^2(Q(R_i,R_i^2))}\right).
\end{equation*}
Therefore, by \eqref{eq-result-from-applied-by-energy-estimate},
\begin{equation}\label{eq-the-second-iteration-inequality-of-X-i-and-Y-i}
Y_{i+1}\leq C{16}^i\left(X_i+Y_i^{1+\kappa}\right), \qquad \forall i\in\Z^+.
\end{equation}
Let
\begin{equation*}
L_i=X_i+Y_i^{1+\kappa}, \qquad \forall i\in\Z^+.
\end{equation*}
Then, by \eqref{eq-the-first-iteration-inequality-of-X-i-and-Y-i} and \eqref{eq-the-second-iteration-inequality-of-X-i-and-Y-i},
\begin{equation}\label{eq-alinged-the-first-stpe-of-union-of-two-iteration-inequalities}
\begin{aligned}
L_{i+1}=X_{i+1}+Y_{i+1}^{1+\kappa}&\leq C16^{i}\left(X_i^{1+\frac{2}{n+2}}+X_i^{\frac{2}{n+2}}Y_i^{1+\kappa}\right)+C^{1+\kappa}16^{i(1+\kappa)}L_i^{1+\kappa}\\
&\leq C^{1+\kappa}16^{i(1+\kappa)}\left[X_i^{1+\frac{2}{n+2}}+X_i^{\frac{2}{n+2}}Y_i^{1+\kappa}+L_i^{1+\kappa}\right].
\end{aligned}
\end{equation}
Note that $X_i\leq L_i$ and $Y_i^{1+\kappa}\leq L_i$. Hence, if $L_i<1$, then
\begin{equation*}
\begin{aligned}
L_{i+1}=X_{i+1}+Y_{i+1}^{1+\kappa}&\leq C^{1+\kappa}16^{i(1+\kappa)}\left[2L_i^{1+\frac{2}{n+2}}+L_i^{1+\kappa}\right]\\
&\leq 2C^{1+\kappa}16^{i(1+\kappa)}L_i^{1+\sigma}, \qquad \qquad \left(\sigma=\min\left\{\kappa, \frac{2}{n+2}\right\}\right).
\end{aligned}
\end{equation*}
If we choose the constant $\rho$ in \eqref{eq-condition-1-of-first-alternatives-for-Holder} sufficiently small that
\begin{equation}\label{initial-condition-of-X-i-and-Y-i-iterations-1}
L_0=X_0+Y_0^{1+\kappa}\leq \left(\frac{1}{2C}\right)^{\frac{1+\kappa}{\sigma}}\left(\frac{1}{16}\right)^{\frac{1+\kappa}{\sigma^2}}
\end{equation}
holds, then 
\begin{equation*}
L_i\leq \left(\frac{1}{2C}\right)^{\frac{(1+\kappa)(1+\sigma)}{\sigma}}\left(\frac{1}{16}\right)^{\frac{(1+\kappa)(1+i\sigma)}{\sigma^2}}, \qquad \forall i\in\Z^+. 
\end{equation*}
Therefore, $X_i$ and $Y_i$ goes to zero as $i\to\infty$ and the lemma follows.
\end{proof}
For the alternative case, we follow the details in \cite{Di} and \cite{KL2}. We suppose that the assumption of Lemma \ref{lem-first-alternative-for-Holder-estimate} is violated, i.e., for every sub -cylinder $Q(R,\theta_0^{-\alpha}R^2)$
\begin{equation}\label{eq-condition-1-of-second-alternatives-for-Holder}
\left|\left\{(x,t)\in Q(R,\theta_0^{-\alpha}R^2):\,\, u^m(x,t)<\mu^-+\frac{\omega}{2^{s_0}}\right\}\right|>\rho\left|Q(R,\theta_0^{-\alpha}R^2)\right|,
\end{equation}
Since
\begin{equation*}
\mu^-+\frac{\omega}{2^{s_0}}\leq\mu^+-\frac{\omega}{2^{s_0}}, 
\end{equation*}
we can rewrite \eqref{eq-condition-1-of-second-alternatives-for-Holder} as
\begin{equation}\label{eq-condition-2-of-second-alternatives-for-Holder}
\left|\left\{(x,t)\in Q(R,\theta_0^{-\alpha}R^2):\,\, u^m(x,t)>\mu^+-\frac{\omega}{2^{s_0}}\right\}\right|\leq(1-\rho)\left|Q(R,\theta_0^{-\alpha}R^2)\right|
\end{equation}
valid for all cylinders 
\begin{equation*}
Q(R,\theta_0^{-\alpha}R^2)\subset Q(R,a_0^{-\alpha}R^2). 
\end{equation*}
Then, by arguments similar to the Lemma 4.2 in \cite{KL2}, we get the following lemma.
\begin{lem}\label{eq-lem-the-first-lemma-when-the-hypothesis-is-violated-124e}
If \eqref{eq-condition-1-of-first-alternatives-for-Holder} is violated, then there exists a time level
\begin{equation*}
t^{\ast}\in\left[-\theta_0^{-\alpha}R^2,-\frac{\rho}{2}\theta_0^{-\alpha}R^2\right]
\end{equation*}
such that
\begin{equation*}
\left|\left\{x\in B_R:\,\, u^m(x,t)>\mu^+-\frac{\omega}{2^{s_0}}\right\}\right|\leq\frac{1-\rho}{1-\frac{\rho}{2}}|B_R|.
\end{equation*}
\end{lem}
The lemma asserts that, at some time $t^{\ast}$, the set where $u^{m}$ is close to its supremum captures only a portion of the $B_R$. The next lemma give us that this occurs for all time levels near the $Q(R,\theta_0^{-\alpha}R^2)$. Set
\begin{equation*}
H=\sup_{B_{R}\times\left[t^{\ast},0\right]}\left|\left(u^m-\left(\mu^+-\frac{\omega}{2^{s_0}}\right)\right)_+\right|.
\end{equation*}
\begin{lem}\label{lem-main-tools-for=the-second-altermativce-estimatieion-3452}
There exists a positive integer $s_1>s_0$ such that if
\begin{equation*}
H>\frac{\omega}{2^{s_1}},
\end{equation*}
then
\begin{equation}\label{eq-in-the-second-lemma-for-the-second-alternative-for-Holder-estiamt}
\left|\left\{x\in B_R:\,\, u^m(x,t)>\mu^+-\frac{\omega}{2^{s_1}}\right\}\right|\leq\left(1-\left(\frac{\rho}{2}\right)^2\right)|B_R|, \qquad \forall t\in\left[t^{\ast},0\right].
\end{equation}
\end{lem}
\begin{proof}
The proof is similar to that of Lemma 7.1 of the Chapter III in \cite{Di}. We introduce the logarithmic function which appears in Section 2 in \cite{Di}
\begin{equation*}
\Psi\left(H,(u^m-k)_+,c\right)\equiv\max\left\{\log\left(\frac{H}{H-\left(u^m-k\right)_++c}\right),0\right\}
\end{equation*}
for $k=\mu^+-\frac{\omega}{2^{s_0}}$, $c=\frac{\omega}{2^{s_1}}$. For simplicity, set $\Psi\left(H,\left(u^m-k\right)_+,c\right)=\vp(u^m)$. We next apply to the first equation of \eqref{eq-cases-aligned-main-problem-of-Keller-Segel-System} the testing  function
\begin{equation*}
mu^{m-1}\frac{\partial}{\partial u^{\ast}}\left[\vp^2(u^{\ast})\right]\xi^2=m\left(u^{\ast}\right)^{\alpha}\left[\vp^2(u^{\ast})\right]'\xi^2, \qquad \left(u^{\ast}=u^m,\,\,\alpha=1-\frac{1}{m}\right) 
\end{equation*}
where $\xi(x)$ is a smooth cut-off function such that, for $0<\nu<1$,
\begin{equation*}
\xi=1 \qquad \mbox{in $B_{(1-\nu)R}$}, \qquad \qquad \xi=0 \qquad \mbox{on $\partial B_{R}$}
\end{equation*}
and 
\begin{equation*}
\left|D\xi\right|\leq \frac{2}{\nu R}.
\end{equation*}
Then, we have for every $t^{\ast}<t<t_0$
\begin{equation*}
\begin{aligned}
\int_{B_R}\left(\vp^2\xi^2\right)_t\,dx&=-\int_{B_R}\nabla u^m\cdot\nabla\left(mu^{m-1}\left(\vp^2\right)'\xi^2\right)\,dx\\
&\qquad \qquad -\int_{B_R}\left(u^{q-1}\nabla v\right)\cdot\nabla\left(mu^{m-1}\left(\vp^2\right)'\xi^2\right)\,dx\\
&\leq -2\left(m-1\right)\int_{B_R}u^{-1}\vp\vp'\xi^2\left|\nabla u^m\right|^2\,dx-2m\int_{B_R}u^{m-1}(1+\vp)\left(\vp'\right)^2\xi^2\left|\nabla u^m\right|^2\,dx\\
&\qquad +4m\int_{B_R}u^{m-1}\vp\left|\vp'\right|\xi\left|\nabla u^m\right|\left|\nabla\xi\right|\,dx+2\left(m-1\right)\int_{B_R}u^{q-2}\vp\left|\vp'\right|\xi^2\left|\nabla v\right|\left|\nabla u^m\right|\,dx\\
&\qquad +2m\int_{B_R}u^{q+m-2}(1+\vp)\left(\vp'\right)^2\xi^2\left|\nabla v\right|\left|\nabla u^m\right|\,dx+4m\int_{B_R}u^{q+m-2}\vp\left|\vp'\right|\xi\left|\nabla v\right|\left|\nabla\xi\right|\,dx
\end{aligned}
\end{equation*}
By Young's inequality,
\begin{equation}\label{eq-after-young-s-inequality-for-the-second-alternative}
\begin{aligned}
\int_{B_R}\left(\vp^2\xi^2\right)_t\,dx&\leq -\left(m-1\right)\int_{B_R}u^{-1}\vp\vp'\xi^2\left|\nabla u^m\right|^2\,dx-m\int_{B_R}u^{m-1}\left(\vp'\right)^2\xi^2\left|\nabla u^m\right|^2\,dx\\
&\qquad +8m\int_{B_R}u^{m-1}\vp\left|\nabla\xi\right|^2\,dx\\
&\qquad +6m\int_{B_R}\left(1+u^m\right)\left(1+\vp\right)\left(\vp'\right)^2u^{2q-3}\xi^2\left|\nabla v\right|^2\,dx.
\end{aligned}
\end{equation}
Note that 
\begin{equation}\label{eq-sign-of-vp-'-1}
\vp'=\frac{1}{H-(u^m-k)_++c}>0
\end{equation}
and by Lemma \ref{lemma-control-the-last-term-by-A-l-r-measure=at=t}
\begin{equation}\label{eq-aligned-control-of-the-last-term-in-term-with-Psi}
\begin{aligned}
&\int_{t^{\ast}}^0\int_{B_R}\left(1+u^m\right)\left(1+\vp\right)\left(\vp'\right)^2u^{2q-3}\xi^2\left|\nabla v\right|^2\,dx\\
&\qquad \qquad \leq \left(1+(\mu^+)^m\right)\left(1+(s_1-s_0)\log2\right)\left(\frac{2^{s_1}}{\omega}\right)^2\left(\mu^+-\frac{\omega}{2^{s_0}}\right)^{2q-3}\left(\int_{t^{\ast}}^0\left|A^+_{1}(t)\right|^{\frac{\tilde{r}}{\tilde{q}}}\,dt\right)^{\frac{2}{\tilde{r}}(1+\kappa)}
\end{aligned}
\end{equation}
where 
\begin{equation*}
\left|A^+_1(t)\right|=\left|\left\{x\in B_R:u(x,t)>\mu^+-\frac{\omega}{2^{s_0}}\right\}\right|.
\end{equation*}
Thus, by \eqref{eq-condition-for-s-0}, \eqref{eq-after-young-s-inequality-for-the-second-alternative}, \eqref{eq-sign-of-vp-'-1}, \eqref{eq-aligned-control-of-the-last-term-in-term-with-Psi} and Lemma \ref{eq-lem-the-first-lemma-when-the-hypothesis-is-violated-124e},
\begin{equation}\label{eq-aligned-after-integration-for-the-time-from-t-0-to-0t-ast-and-simplification--0}
\begin{aligned}
&\sup_{t^{\ast}<t<0}\int_{B_R}\Psi^2(H,(u^m-k)_-,c)(x,t)\xi^2(x)\,dx\\
&\qquad \leq (s_1-s_0)^2\left(\log 2\right)^2\left(\frac{1-\rho}{1-\frac{\rho}{2}}\right)|B_R|+\frac{8m\left(2^{s_0}\right)^{\alpha}\left(\mu^+\right)^{m-1}(s_1-s_0)\log 2 }{\omega^{\alpha}\nu^2}|B_R|\\
& \qquad \qquad+6m\left(1+(\mu^+)^{m}\right)\left(1+(s_1-s_0)\log 2\right)\left(\frac{2^{s_1}}{\omega}\right)^2\left(\mu^+-\frac{\omega}{2^{s_0}}\right)^{2q-3}\frac{R^{n\kappa}}{\theta_0^{\alpha\left(1-\frac{1}{a_2}\right)}}|B_R|
\end{aligned}
\end{equation} 
where $a_2$ is given in \eqref{eq-relation-between-a-1-and-a-2-for-smooth-computaiotion}. If we choose $\epsilon$ small enough that $\epsilon\left(1-\frac{1}{a_2}\right)<n\kappa$, then by \eqref{eq-condition-between-R-and-theta-alpha} and \eqref{eq-aligned-after-integration-for-the-time-from-t-0-to-0t-ast-and-simplification--0},
\begin{equation}\label{eq-aligned-after-integration-for-the-time-from-t-0-to-0t-ast-and-simplification-1}
\begin{aligned}
&\sup_{t^{\ast}<t<0}\int_{B_R}\Psi^2(H,(u^m-k)_-,c)(x,t)\xi^2(x)\,dx\\
&\qquad \leq (s_1-s_0)^2\left(\log 2\right)^2\left(\frac{1-\rho}{1-\frac{\rho}{2}}\right)|B_R|+\frac{8m\left(2^{s_0}\right)^{\alpha}\left(\mu^+\right)^{m-1}(s_1-s_0)\log 2 }{\omega^{\alpha}\nu^2}|B_R|\\
& \qquad \qquad+6m\left(1+(\mu^+)^{m}\right)\left(1+(s_1-s_0)\log 2\right)\left(\frac{2^{s_1}}{\omega}\right)^2\left(\mu^+-\frac{\omega}{2^{s_0}}\right)^{2q-3}\left(\frac{2^{s_0}}{A}\right)^{\alpha\left(1-\frac{1}{a_2}\right)}|B_R|.
\end{aligned}
\end{equation} 
The integral on the left hand side of \eqref{eq-aligned-after-integration-for-the-time-from-t-0-to-0t-ast-and-simplification-1} is bounded below by integrating over the small set
\begin{equation*}
\left\{x\in B_{(1-\nu)R}:u^m(x,t)>\mu^+-\frac{\omega}{2^{s_1}}\right\}.
\end{equation*}
On such a set
\begin{equation}\label{eq-lower-bound-of-Psi-over-B-1---mu-R}
\Psi\left(H,\left(u^m-\left(\mu^+-\frac{\omega}{2^{s_0}}\right)\right),\frac{\omega}{2^{s_1}}\right)\geq (s_1-s_0-1)^2\left(\log 2\right)^2.
\end{equation}
Thus, by \eqref{eq-construction-of-cylinder-with-some-constant-A-which-is-bigger}, \eqref{eq-aligned-after-integration-for-the-time-from-t-0-to-0t-ast-and-simplification-1} and \eqref{eq-lower-bound-of-Psi-over-B-1---mu-R} , we get
\begin{equation*}
\begin{aligned}
&\left|\left\{x\in B_{(1-\nu)R}:u^m(x,t)>\mu^+-\frac{\omega}{2^{s_1}}\right\}\right|\\
&\qquad \leq \left(\frac{s_1-s_0}{s_1-s_0-1}\right)^2\left(\frac{1-\rho}{1-\frac{\rho}{2}}\right)|B_R|+\frac{8m\left(2^{s_0}\right)^{\alpha}(\mu^+)^{m-1}(s_1-s_0)}{\nu^2\omega^{\alpha}(s_1-s_0-1)^2\log 2}|B_R|\\
& \qquad \qquad +6m\left(1+(\mu^+)^{m}\right)\left(\frac{1+(s_1-s_0)\log 2}{(s_1-s_0-1)^2\left(\log2\right)^2}\right)\left(\frac{2^{s_1}}{\omega}\right)^2\left(\mu^+-\frac{\omega}{2^{s_0}}\right)^{2q-3}\left(\frac{2^{s_0}}{A}\right)^{\alpha\left(1-\frac{1}{a_2}\right)}|B_R|.
\end{aligned}
\end{equation*}
On the other hand, 
\begin{equation*}
\left|\left\{x\in B_{R}:u^m(x,t)>\mu^+-\frac{\omega}{2^{s_1}}\right\}\right|\leq \left|\left\{x\in B_{(1-\nu)R}:u^m(x,t)>\mu^+-\frac{\omega}{2^{s_1}}\right\}\right|+n\nu|B_R|.
\end{equation*}
Therefore,
\begin{equation}\label{eq-aligned-after-integration-for-the-time-from-t-0-to-0t-ast-and-simplification-2}
\begin{aligned}
&\left|\left\{x\in B_{R}:u^m(x,t)>\mu^+-\frac{\omega}{2^{s_1}}\right\}\right|\\
&\qquad \leq \left(\frac{s_1-s_0}{s_1-s_0-1}\right)^2\left(\frac{1-\rho}{1-\frac{\rho}{2}}\right)|B_R|+n\nu|B_R|+\frac{8m\left(2^{s_0}\right)^{\alpha}(\mu^+)^{m-1}(s_1-s_0)}{\nu^2\omega^{\alpha}(s_1-s_0-1)^2\log 2}|B_R|\\
& \qquad \qquad +6m\left(1+(\mu^+)^{m}\right)\left(\frac{1+(s_1-s_0)\log 2}{(s_1-s_0-1)^2\left(\log2\right)^2}\right)\left(\frac{2^{s_1}}{\omega}\right)^2\left(\mu^+-\frac{\omega}{2^{s_0}}\right)^{2q-3}\left(\frac{2^{s_0}}{A}\right)^{\alpha\left(1-\frac{1}{a_2}\right)}|B_R|.
\end{aligned}
\end{equation}
To prove the lemma, we choose $\nu$ so small that $n\nu\leq\frac{1}{4}\rho^2$ and $s_1$ so large that
\begin{equation*}
\left(\frac{s_2-s_1}{s_2-s_1-1}\right)^2\leq \left(1-\frac{1}{2}\rho\right)(1+\rho) \qquad \mbox{and} \qquad \frac{8m\left(2^{s_0}\right)^{\alpha}(\mu^+)^{m-1}(s_1-s_0)}{\nu^2\omega^{\alpha}(s_1-s_0-1)^2\log 2}\leq \frac{1}{4}\rho^2.
\end{equation*}
For such $\nu$ and $s_1$, we can also choose $A$ in \eqref{eq-construction-of-cylinder-with-some-constant-A-which-is-bigger} sufficiently large that
\begin{equation*}
6m\left(1+(\mu^+)^{m}\right)\left(\frac{1+(s_1-s_0)\log 2}{(s_1-s_0-1)^2\left(\log2\right)^2}\right)\left(\frac{2^{s_1}}{\omega}\right)^2\left(\mu^+-\frac{\omega}{2^{s_0}}\right)^{2q-3}\left(\frac{2^{s_0}}{A}\right)^{\alpha\left(1-\frac{1}{a_2}\right)}\leq \frac{1}{4}\rho^2.
\end{equation*}
Then, the inequality \eqref{eq-in-the-second-lemma-for-the-second-alternative-for-Holder-estiamt} holds and the lemma follows.
\end{proof}
\begin{corollary}[cf. Corollary 7.1 of Chapter III in \cite{Di}]\label{cor-cf-Corollary-7-1-of-Chapter-III-in-cite-Di}
For all $t\in \left(-\frac{R^2}{2a_0^{\alpha}},0\right)$,
\begin{equation}\label{eq-in-the-second-lemma-for-the-second-alternative-for-Holder-estiamt-in-a-0---ahpah}
\left|\left\{x\in B_R:\,\, u^m(x,t)>\mu^+-\frac{\omega}{2^{s_1}}\right\}\right|\leq\left(1-\left(\frac{\rho}{2}\right)^2\right)|B_R|..
\end{equation}
\end{corollary}
By an argument similar to the proof of Lemma 8.1 of Chapter III in \cite{Di}, we have the following lemma. 
\begin{lem}\label{lemma-complete-assumption-for-the-second-alternative}
If \eqref{eq-condition-1-of-first-alternatives-for-Holder} is violated,  for every $\nu_{\ast}\in(0,1)$, there exists a number $s^{\ast}>s_1+1>s_0$ independent of $\omega$ and $R$ such that
\begin{equation*}
\left|\left\{(x,t)\in Q\left(R,\frac{1}{2}a_0^{-\alpha}R^2\right): u^m(x,t)>\mu^+-\frac{\omega}{2^{s^{\ast}}}\right\}\right|\leq \nu_{\ast}\left|Q\left(R,\frac{1}{2}a_0^{-\alpha}R^2\right)\right|
\end{equation*}
with the constant $A=2^{s^{\ast}}$.
\end{lem}
By Lemma \ref{lemma-complete-assumption-for-the-second-alternative}, $\nu_{\ast}$ decided a level and a cylinder so that the measure of the set where $u^m$ is above such a level can be smaller than $\nu_{\ast}$ on that particular cylinder. Hence, for sufficiently small number $\nu_{\ast}$ we have a similar assumption to the one in Lemma \ref{lem-first-alternative-for-Holder-estimate}. Therefore, by an argument similar to the proof of Lemma \ref{lem-first-alternative-for-Holder-estimate} with 
\begin{equation*}
\left(u^m-\left(\mu^+-\frac{\omega}{2^{s^{\ast}}}\right)\right)_+,
\end{equation*}
we can have the following result
\begin{lem}\label{lem-powerful-hypothesis-for-teh-secind-alternatives-o}
The number $\nu_{\ast}$ can be chosen so that 
\begin{equation*}
u^{m}(x,t)\leq \mu^{+}-\frac{\omega}{2^{s^{\ast}+1}} \qquad \mbox{a.e. $Q\left(\frac{R}{2},\frac{1}{2}a_0^{-\alpha}\left(\frac{R}{2}\right)^2\right)$}.
\end{equation*}
\end{lem}
Combining Lemma \ref{lem-first-alternative-for-Holder-estimate} with Lemma \ref{lem-powerful-hypothesis-for-teh-secind-alternatives-o}, we can obtain the following Osillation Lemma.
\begin{lem}[Oscillation Lemma]\label{lem-Oscillation-Lemma}
There exists constant $0<\lambda^{\ast}<1$ such that if 
\begin{equation*}
\osc_{Q_{R}}u^m=\omega=\mu^+-\mu^-,
\end{equation*}
then
\begin{equation*}
\osc_{Q\left(\frac{R}{2},\frac{1}{2}a_0^{-\alpha}\left(\frac{R}{2}\right)^2\right)}u^m\leq \lambda^{\ast}\omega.
\end{equation*}
\end{lem}
\begin{thm}[H\"older estimates]\label{thm-Holder-estimates-for-solution-u}
There exists constant $\lambda^{\star}>1$ and $\beta\in(0,1)$ that can be determined a priori only in terms of the data, such that for all the cylinders
\begin{equation*}
\osc_{Q\left(r,\frac{1}{2}a_0^{-\alpha}r^2\right)}u^m\leq \lambda^{\star}\omega\left(\frac{r}{R}\right)^{\beta} \qquad \left(0<r\leq R\right).
\end{equation*}
\end{thm}
\begin{proof}
The proof is very similar to the proof of Theorem 4.10 of \cite{KL2}. For future references we will sketch the proof of the H\"older estimates. Let $k$ be positive integer. By the Oscillation Lemma (Lemma \ref{lem-Oscillation-Lemma}), we get
\begin{equation*}
\osc_{Q\left(\frac{R}{2^k},\frac{1}{2}a_0^{-\alpha}\left(\frac{R}{2^k}\right)^2\right)}u^m\leq \left(\lambda^{\ast}\right)^k\omega, \qquad \left(\lambda^{\ast}<1\right).
\end{equation*}
Let $0<r\leq R$ be fixed. Then, there is a non-negative integer $k$ such that
\begin{equation*}
\frac{R}{2^{k+1}}< r\leq \frac{R}{2^{k}}.
\end{equation*}
This immediately implies the inequalities
\begin{equation*}
k\leq -\log_2\left(\frac{r}{R}\right)<k+1 \qquad \mbox{and} \qquad \left(\lambda^{\ast}\right)^k\leq \frac{1}{\lambda^{\ast}}\left(\lambda^{\ast}\right)^{-\log_2\left(\frac{r}{R}\right)}=\frac{1}{\lambda^{\ast}}\left(\frac{r}{R}\right)^{-\log_2\lambda^{\ast}}.
\end{equation*}
Note that
\begin{equation*}
\left(1-\frac{1}{2^{s_{0}+1}}\right)\leq \lambda^{\ast}\leq \left(1-\frac{1}{2^{s^{\ast}+1}}\right).
\end{equation*}
Hence
\begin{equation*}
\osc_{Q\left(\frac{R}{2^k},\frac{1}{2}a_0^{-\alpha}\left(\frac{R}{2^k}\right)^2\right)}u^m\leq \lambda^{\star}\omega\left(\frac{r}{R}\right)^{\beta}
\end{equation*}
where $\lambda^{\star}=\frac{1}{\lambda^{\ast}}>1$ and $0<\beta=-\log_2\lambda^{\ast}<1$. To complete the proof, we observe that the cylinder $Q\left(r,\frac{1}{2}a_0^{-\alpha}r^2\right)$ is included in $Q\left(\frac{R}{2^k},\frac{1}{2}a_0^{-\alpha}\left(\frac{R}{2^k}\right)^2\right)$.
\end{proof}

\section{Uniqueness}\label{sec-uniqueness}
\setcounter{equation}{0}
\setcounter{thm}{0}

In this section we will prove that under some conditions on the weak solutions $u$ and $v$ of \eqref{eq-cases-aligned-main-problem-of-Keller-Segel-System}, there exists  at most one weak solution of \eqref{eq-cases-aligned-main-problem-of-Keller-Segel-System} on $[0,\infty)$.  We first start with the well-known lemma.

\begin{lem}[cf. Theorem 6 in \cite{Wa}]\label{lem-cf-Thm-6-in-Wa}
Let $Q_r=B_r(0)\times(-r^2,0]$ be the parabolic cube. If
\begin{equation*}
w_t-\La w=f\in L^{p} \qquad \mbox{in $Q_1$},
\end{equation*}
then
\begin{equation*}
w_t, \,\, D^2w\in L^{p}(Q_{\frac{1}{2}})
\end{equation*}
and
\begin{equation*}
\left\|w_t\right\|_{L^{p}(Q_{\frac{1}{2}})}+\left\|D^2u\right\|_{L^{p}(Q_{\frac{1}{2}})}\leq C\left(\left\|f\right\|_{L^p(Q_1)}+\left\|w\right\|_{L^p(Q_1)}\right).
\end{equation*}
\end{lem}

We finish this work with stating the following result.
\begin{thm}[Uniqueness of weak solution]
Let $m>1$, $q>\max\left(\frac{m}{2}+1,2\right)$, $\gamma>0$. Assume that initial data $u_0\in L^1\cap L^{\infty}(\R^n)$ and $v_0\in L^1(\R^n)$ with $\nabla v_0\in L^{\infty}(\R^n)$ satisfy the following additional conditions
\begin{equation*}
u_0\in C^{\alpha}(\R^n), \qquad v_0\in C^{2,\alpha}(\R^n) \qquad \mbox{for some $0<\alpha<1$}.
\end{equation*}
If $u$ is a weak solution of \eqref{eq-cases-aligned-main-problem-of-Keller-Segel-System} satisfying the properties
\begin{equation*}
u_t\in L^1(0,T;L^1_{loc}(\R^n)), \qquad u(\cdot,t)\in C(\R^n) \quad \mbox{for a.e. $0<t<T$}
\end{equation*}
and
\begin{equation*}
u\in L^{q-1}(0,T;L^{\infty}(\R^n))\cap L^{q-m-1}(0,T;L^{\infty}(\R^n))\cap L^{q-m}(0,T;L^{\infty}(\R^n))\cap L^{m}(0,T;L^m(\R^n)).
\end{equation*}
In addition, we assume the following alternatives
\begin{enumerate} 
\item In the case of $1<m<2$ and in the case of $m\geq 2$ and $q\leq m+1$,
\begin{equation*}
u\in L^2\left(0,T; L^{\frac{2n}{n+2}}(\R^n)\right);
\end{equation*}
\item In the case of $m\geq 2$ and $q>m+1$,
\begin{equation*}
u\in L^{2q-m-1}(0,T;L^{\infty(\R^n)}).
\end{equation*}
\end{enumerate}
Then, the weak solution $(u,v)$ of \eqref{eq-cases-aligned-main-problem-of-Keller-Segel-System} on $[0,T)$ is unique.
\end{thm}
\begin{proof}
Note that the proof of \textbf{Case 1} is very similar to the proof of Theorem 2.1 of \cite{SY}. Hence we only need to prove the \textbf{Case 2}. We will use a modification of the technique of \cite{SY} to prove the theorem. Since $(u,v)$ is a weak solution of \eqref{eq-cases-aligned-main-problem-of-Keller-Segel-System} on $[0,T)$, we have 
\begin{equation}\label{eq-equation-satisfied-by-weak-solutions-with-testt-functions-1}
\int_{t_1}^{t_2}\int_{\R^n}\left(\partial_{\tau}u(x,\tau)\cdot\phi(x,\tau)+\nabla u^m(x,\tau)\cdot\nabla\phi(x,\tau)-u^{q-1}\nabla v(x,\tau)\cdot\nabla\phi(x,\tau)\right)\,dxd\tau=0
\end{equation}
for all $0<t_1<t_2<T$ and for all $\phi\in L^2\left(0,T;H^1(\R^n)\right)\cap C\left((0,T);L^{\infty}(\R^n)\right)$ compactly supported in $\R^n$ for all $t\in[0,T)$ with $\phi(\cdot,0)\in L^{\infty}(\R^n)$.

Let $\eta\in C^1(\R)$ be such that $0\leq \eta(s)\leq 1$ for all $s\in\R$ and $0<\eta'(s)<2$ for all $s>0$ and 
\begin{equation*}
\eta(s)=\begin{cases}
                  \begin{array}{c}
                     0 \qquad \mbox{for $s\leq 0$}\\
                    1 \qquad \mbox{for $s\geq 1$}.
                    \end{array}
                 \end{cases}
\end{equation*}
Let $(u,v)$ and $(\hat{u},\hat{v})$ be  two weak solutions of \ref{eq-cases-aligned-main-problem-of-Keller-Segel-System} on $[0,T)$. We define $\eta_k(r)=\eta(kr)$ for all $r\in\R$ and all $n=1,2,\cdots$.

\begin{equation*}
\phi:=\eta_k\left(u^m-\hat{u}^m\right)\cdot\psi_l\in L^2(0,T;H^1(\R^n))\cap C\left((0,T);L^{\infty}\left(\R^n\right)\right)
\end{equation*}
with $\phi(\cdot,0)\in L^{\infty}(\R^n)$. It can be easily checked that $\supp\phi(\cdot,t)$ is compact in $\R^n$ for all $t\in[0,T)$. As the test function, we apply this $\phi$ to \eqref{eq-equation-satisfied-by-weak-solutions-with-testt-functions-1}. Then, for all $0<t_1<t_2<T$,
\begin{equation}\label{eq-splitting-quantity-into-two-I-and-II}
\int_{t_1}^{t_2}\int_{\R^n}\eta_k\left(u^m-\hat{u}^m\right)\psi_l\cdot\partial_{\tau}\left(u(x,\tau)-\hat{u}(x,\tau)\right)\,dxd\tau:=I+II
\end{equation}
where
\begin{equation*}
I:=-\int_{t_1}^{t_2}\int_{\R^n}\nabla\left(u^m-\hat{u}^m\right)\cdot\nabla\left(\eta_k\left(u^m-\hat{u}^m\right)\cdot\psi_l\right)\,dxd\tau,
\end{equation*}
and
\begin{equation*}
II:=\int_{t_1}^{t_2}\int_{\R^n}\left(u^{q-1}\nabla v-\hat{u}^{q-1}\nabla\hat{v}\right)\cdot\nabla\left(\eta_k\left(u^m-\hat{u}^m\right)\cdot\psi_l\right)\,dxd\tau.
\end{equation*}
By the chain rule,
\begin{equation*}
\begin{aligned}
I&=-\int_{t_1}^{t_2}\int_{\R^n}\left|\nabla\left(u^m-\hat{u}^m\right)\right|^2\cdot\eta_k'\left(u^m-\hat{u}^m\right)\cdot\psi_l\,dxd\tau\\
&\qquad \qquad \qquad -\int_{t_1}^{t_2}\int_{\R^n}\eta_k\left(u^m-\hat{u}^m\right)\nabla\left(u^m-\hat{u}^m\right)\cdot\nabla\psi_l\,dxd\tau.
\end{aligned}
\end{equation*}
Define a domain $D_k$ by
\begin{equation*}
D_k:=\left\{(x,t)\in\R^n\times(0,T):0<u^m(x,t)-\hat{u}^{m}(x,t)<\frac{1}{k}\right\}.
\end{equation*}
Then, by Integration by parts and Young's inequality,
\begin{equation*}
\begin{aligned}
I&=-\frac{3}{4}\int_{t_1}^{t_2}\int_{\R^n}\left|\nabla\left(u^m-\hat{u}^m\right)\right|^2\cdot\eta_k'\left(u^m-\hat{u}^m\right)\cdot\psi_l\,dxd\tau\\
&\qquad \qquad +\int_{t_1}^{t_2}\int_{D_k\cap\supp \psi_{l}}\left|u^m-\hat{u}^m\right|^2\eta_k'\left(u^m-\hat{u}^m\right)\cdot\frac{\left|\nabla\psi_l\right|^2}{\psi_l}\,dxd\tau\\
&\qquad \qquad \qquad \qquad + \int_{t_1}^{t_2}\int_{\R^n} \left(u^m-\hat{u}^m\right)\cdot \eta_k\left(u^m-\hat{u}^m\right)\cdot\La\psi_l\,dxd\tau.
\end{aligned}
\end{equation*}
Thus,
\begin{equation}\label{eq-control-of-I-by-some-quantities}
I\leq -\frac{3}{4}I_1+\frac{2c_1^2}{kl^2}\int_{t_1}^{t_2}\int_{D_k\cap\supp \psi_{l}}\,dxd\tau+\frac{c_2}{l^2}\int_{t_1}^{t_2}\int_{\R^n} \left(u^m+\hat{u}^m\right)\,dxd\tau
\end{equation}
where
\begin{equation*}
I_1:=\int_{t_1}^{t_2}\int_{\R^n}\left|\nabla\left(u^m-\hat{u}^m\right)\right|^2\cdot\eta_k'\left(u^m-\hat{u}^m\right)\cdot\psi_l\,dxd\tau.
\end{equation*}
We next pay our attention to $II$. By simple computations,
\begin{equation}\label{eq-aligned-control-of-II-by-some-quantities}
\begin{aligned}
II&=-\int_{t_1}^{t_2}\int_{\R^n}\nabla\left(u^{q-1}\nabla v-\hat{u}^{q-1}\nabla\hat{v}\right)\cdot\eta_k\left(u^m-\hat{u}^m\right)\cdot\psi_l\,dxd\tau\\
&=-\int_{t_1}^{t_2}\int_{\R^n}\left(\nabla u^{q-1}-\nabla \hat{u}^{q-1}\right)\cdot\nabla\hat{v}\cdot\eta_k\left(u^m-\hat{u}^m\right)\cdot\psi_l\,dxd\tau\\
&\qquad -\int_{t_1}^{t_2}\int_{\R^n}\nabla u^{q-1}\cdot\left(\nabla v-\nabla\hat{v}\right)\cdot\eta_k\left(u^m-\hat{u}^m\right)\cdot\psi_l\,dxd\tau\\
&\qquad -\int_{t_1}^{t_2}\int_{\R^n}\left(u^{q-1}\cdot\left(\La v-\La\hat{v}\right)+\left(u^{q-1}-\hat{u}^{q-1}\right)\La\hat v\right)\cdot\eta_k\left(u^m-\hat{u}^m\right)\cdot\psi_l\,dxd\tau\\
&=II_1+II_2+II_3.
\end{aligned}
\end{equation}
By an argument similar to the (4.16) of \cite{SY}
\begin{equation}\label{eq-aligned-control-of-II-1}
\begin{aligned}
II_1&\leq C\int_{t_1}^{t_2}\left(\left\|u(\tau)\right\|_{L^{\infty}}^{q-1}+\left\|\hat{u}(\tau)\right\|_{L^{\infty}}^{q-1}\right)\left\|\left(u-\hat{u}\right)\right\|_{L^{1}}\,d\tau\\
&\quad +\frac{I_1}{4}+\frac{C\chi_{m,q}}{k^{\frac{2(q-1)}{m}-1}}\int_{t_1}^{t_2}\left(\left\|\nabla\hat{v}_0\right\|_{L^{\infty}}^2+\left\|\hat{u}(\tau)\right\|_{L^{\frac{2n}{n+2}}}\right)\,d\tau \\
&\qquad +C\left(1-\chi_{m,q}\right)\int_{t_1}^{t_2}\left(\left\|v_0\right\|^{2q-m-1}_{L^{\infty}}+\left\|u(\tau)\right\|^{2q-m-1}_{L^{\infty}}+\left\|\hat{u}(\tau)\right\|^{2q-m-1}_{L^{\infty}}\right)\left\|(u-\hat{u})(\tau)\right\|_{L^1}\,d\tau
\end{aligned}
\end{equation}
for all $0<t_1<t_2<T$, where $C=C(n,M,m,q,\gamma)$.\\
\indent By Lemma \ref{lem-Lemma-5-in-SK}, it holds that
\begin{equation}\label{eq-aligned-control-of-II-2--1}
\begin{aligned}
II_2&\leq \frac{q-1}{m}\int_{t_1}^{t_2}\left\|u(\tau)\right\|_{L^{\infty}}^{q-m-1}\left\|\nabla u^{m}(\tau)\right\|_{L^2}\cdot\left\|\left(\nabla v-\nabla\hat{v}\right)(\tau)\right\|_{L^{2}}\,d\tau\\
&\leq C\int_{t_1}^{t_2}\left\|u(\tau)\right\|_{L^{\infty}}^{q-m-1}\left\|\nabla u^{m}(\tau)\right\|_{L^2}\cdot\left\|\left(u-\hat{u}\right)(\tau)\right\|_{L^{2}}\,d\tau
\end{aligned}
\end{equation}
for all $0<t_1<t_2<T$ with some positive constant where $C=C(p,n)$. Applying Young's inequality in \eqref{eq-aligned-control-of-II-2--1}, we get
\begin{equation}\label{eq-aligned-control-of-II-2}
\begin{aligned}
II_2\leq C\epsilon\int_{t_1}^{t_2}\left\|u(\tau)\right\|_{L^{\infty}}^{q-m-1}\left\|\nabla u^{m}(\tau)\right\|^2_{L^2}\,d\tau+\frac{C}{\epsilon}\int_{t_1}^{t_2}\left(\left\|u(\tau)\right\|_{L^{\infty}}^{q-m}+\left\|\hat{u}(\tau)\right\|_{L^{\infty}}^{q-m}\right)\left\|\left(u-\hat{u}\right)(\tau)\right\|_{L^{1}}\,d\tau
\end{aligned}
\end{equation}
for any $\epsilon>0$.\\
\indent By H\"older estimates (Theorem \ref{thm-Holder-estimates-for-solution-u}), $u, \hat{u}\in C^{\beta,\frac{\beta}{2}}(\R^n,(0,\infty))$ for some constant $0<\beta<1$. Then, by standard Schauder's estimates for the heat equation, $\La v$ and $\La\hat{v}$ are also H\"older continuous in space and time. Hence, by Lemma \ref{lem-Lemma-5-in-SK} and Lemma \ref{lem-cf-Thm-6-in-Wa}, we have
\begin{equation}\label{eq-aligned-control-of-II-3}
\begin{aligned}
II_3&\leq C\bigg(\int_{t_1}^{t_2}\left\|u(\tau)\right\|^{q-1}_{L^{\infty}}\left(\left\|(v-\hat{v})(\tau)\right\|_{L^{1}}+\left\|\left(u-\hat{u}\right)(\tau)\right\|_{L^{1}}\right)\,d\tau\\
&\qquad \qquad +\int_{t_1}^{t_2}\left(\left\|u(\tau)\right\|^{q-2}_{L^{\infty}}+\left\|\hat{u}(\tau)\right\|^{q-2}_{L^{\infty}}\right)\left(\left\|\hat{v}(\tau)\right\|_{L^{1}}+\left\|\hat{u}(\tau)\right\|_{L^1}\right)\left\|\left(u-\hat{u}\right)(\tau)\right\|_{L^{1}}\,d\tau\bigg)\\
&\leq C\left(\int_{t_1}^{t_2}\left(\left\|u(\tau)\right\|^{q-1}_{L^{\infty}}+\left\|\hat{u}(\tau)\right\|^{q-1}_{L^{\infty}}\right)\left\|\left(u-\hat{u}\right)(\tau)\right\|_{L^{1}}\,d\tau\right)
\end{aligned}
\end{equation}
where $C=C(n,q,\gamma)$.\\
\indent Substituting \eqref{eq-aligned-control-of-II-1}, \eqref{eq-aligned-control-of-II-2} and \eqref{eq-aligned-control-of-II-3} into \eqref{eq-aligned-control-of-II-by-some-quantities}, we have
\begin{equation}\label{eq-aligned-control-of-II-by-summations}
\begin{aligned}
II&\leq \frac{I_1}{4}+\frac{Ck\chi_{m,q}}{k^{\frac{2(q-1)}{m}}}\int_{t_1}^{t_2}\left(\left\|\nabla\hat{v}_0(\tau)\right\|_{L^{\infty}}^2+\left\|\hat{u}(\tau)\right\|^2_{L^{\frac{2n}{n+2}}}\right)\int_{D_k}\psi_l\,dxd\tau\\
&\qquad +C\left(1-\chi_{m,q}\right)\int_{t_1}^{t_2}\left(\left\|v_0\right\|^{2q-m-1}_{L^{\infty}}+\left\|u(\tau)\right\|^{2q-m-1}_{L^{\infty}}+\left\|\hat{u}(\tau)\right\|^{2q-m-1}_{L^{\infty}}\right)\left\|(u-\hat{u})(\tau)\right\|_{L^1}\,d\tau\\
&\qquad +\frac{C}{\epsilon}\int_{t_1}^{t_2}\left(\left\|u(\tau)\right\|_{L^{\infty}}^{q-m}+\left\|\hat{u}(\tau)\right\|_{L^{\infty}}^{q-m}\right)\left\|\left(u-\hat{u}\right)(\tau)\right\|_{L^{1}}\,d\tau\\
&\qquad +C\int_{t_1}^{t_2}\left(\left\|u(\tau)\right\|^{q-1}_{L^{\infty}}+\left\|\hat{u}(\tau)\right\|^{q-1}_{L^{\infty}}\right)\left\|\left(u-\hat{u}\right)(\tau)\right\|_{L^{1}}\,d\tau\\
&\qquad +C\epsilon\int_{t_1}^{t_2}\left\|u(\tau)\right\|_{L^{\infty}}^{q-m-1}\left\|\nabla u^{m}(\tau)\right\|^2_{L^2}\,d\tau
\end{aligned}
\end{equation}
for all $0<t_1<t_2<T$ where $C=C(p,n,m,q,\gamma)$.\\
\indent Putting \eqref{eq-control-of-I-by-some-quantities} and \eqref{eq-aligned-control-of-II-by-summations} into \eqref{eq-splitting-quantity-into-two-I-and-II},
\begin{equation}\label{eq-aligned-control-of-difference-and-test-functions-in-L-1-at-time-t}
\begin{aligned}
&\int_{t_1}^{t_2}\int_{\R^n}\eta_k\left(u^m-\hat{u}^m\right)\cdot\psi_l\cdot\partial_{\tau}\left(u(x,\tau)-\hat{u}(x,\tau)\right)\,dxd\tau\\
&\qquad \qquad \leq \frac{2c_1^2}{k^2l^2}\int_{t_1}^{t_2}\int_{D_k\cap\supp \psi_{l}}\,dxd\tau+\frac{c_2}{l^2}\int_{t_1}^{t_2}\int_{\R^n} \left(u^m+\hat{u}^m\right)\,dxd\tau\\
&\qquad \qquad \qquad +\frac{Ck\chi_{m,q}}{k^{\frac{2(q-1)}{m}}}\int_{t_1}^{t_2}\left(\left\|\nabla\hat{v}_0(\tau)\right\|_{L^{\infty}}^2+\left\|\hat{u}(\tau)\right\|^2_{L^{\frac{2n}{n+2}}}\right)\int_{D_k}\psi_l\,dxd\tau\\
&\qquad \qquad \qquad +C\left(1-\chi_{m,q}\right)\int_{t_1}^{t_2}\left(\left\|v_0\right\|^{2q-m-1}_{L^{\infty}}+\left\|u(\tau)\right\|^{2q-m-1}_{L^{\infty}}+\left\|\hat{u}(\tau)\right\|^{2q-m-1}_{L^{\infty}}\right)\left\|(u-\hat{u})(\tau)\right\|_{L^1}\,d\tau\\
&\qquad \qquad \qquad +\frac{C}{\epsilon}\int_{t_1}^{t_2}\left(\left\|u(\tau)\right\|_{L^{\infty}}^{q-m}+\left\|\hat{u}(\tau)\right\|_{L^{\infty}}^{q-m}\right)\left\|\left(u-\hat{u}\right)(\tau)\right\|_{L^{1}}\,d\tau\\
&\qquad \qquad \qquad +C\int_{t_1}^{t_2}\left(\left\|u(\tau)\right\|^{q-1}_{L^{\infty}}+\left\|\hat{u}(\tau)\right\|^{q-1}_{L^{\infty}}\right)\left\|\left(u-\hat{u}\right)(\tau)\right\|_{L^{1}}\,d\tau\\
&\qquad \qquad \qquad +C\epsilon\int_{t_1}^{t_2}\left\|u(\tau)\right\|_{L^{\infty}}^{q-m-1}\left\|\nabla u^{m}(\tau)\right\|^2_{L^2}\,d\tau
\end{aligned}
\end{equation}
By Lebesgue dominated convergence theorem, 
\begin{equation*}
\begin{aligned}
&\int_{t_1}^{t_2}\int_{\R^n}\eta_k\left(u^m-\hat{u}^m\right)\cdot\psi_l\cdot\partial_{\tau}\left(u(x,\tau)-\hat{u}(x,\tau)\right)\,dxd\tau\\
&\qquad \qquad \to \int_{t_1}^{t_2}\int_{\R^n}\psi_l\cdot\partial_{\tau}\left[u(x,\tau)-\hat{u}(x,\tau)\right]_+\,dxd\tau \qquad \qquad \mbox{as $k\to\infty$}.
\end{aligned}
\end{equation*}
Since $\frac{2(q-1)}{m}>1$, by letting $k\to\infty$ in \eqref{eq-aligned-control-of-difference-and-test-functions-in-L-1-at-time-t},
\begin{equation}\label{eq-aligned-control-of-difference-in-L-1-at-time-t}
\begin{aligned}
&\int_{\R^n}\psi_l\cdot\left[u(x,t_2)-\hat{u}(x,t_2)\right]_+\,dx\\
&\qquad \qquad \leq \int_{\R^n}\psi_l\cdot\left[u(x,t_1)-\hat{u}(x,t_1)\right]_+\,dx+\frac{c_2}{l^2}\int_{t_1}^{t_2}\int_{\R^n} \left(u^m+\hat{u}^m\right)\,dxd\tau\\
&\qquad \qquad\qquad +C\left(1-\chi_{m,q}\right)\int_{t_1}^{t_2}\left(\left\|v_0\right\|^{2q-m-1}_{L^{\infty}}+\left\|u(\tau)\right\|^{2q-m-1}_{L^{\infty}}+\left\|\hat{u}(\tau)\right\|^{2q-m-1}_{L^{\infty}}\right)\left\|(u-\hat{u})(\tau)\right\|_{L^1}\,d\tau\\
&\qquad \qquad \qquad +\frac{C}{\epsilon}\int_{t_1}^{t_2}\left(\left\|u(\tau)\right\|_{L^{\infty}}^{q-m}+\left\|\hat{u}(\tau)\right\|_{L^{\infty}}^{q-m}\right)\left\|\left(u-\hat{u}\right)(\tau)\right\|_{L^{1}}\,d\tau\\
&\qquad \qquad \qquad +C\int_{t_1}^{t_2}\left(\left\|u(\tau)\right\|^{q-1}_{L^{\infty}}+\left\|\hat{u}(\tau)\right\|^{q-1}_{L^{\infty}}\right)\left\|\left(u-\hat{u}\right)(\tau)\right\|_{L^{1}}\,d\tau\\
&\qquad \qquad \qquad +C\epsilon\int_{t_1}^{t_2}\left\|u(\tau)\right\|_{L^{\infty}}^{q-m-1}\left\|\nabla u^{m}(\tau)\right\|^2_{L^2}\,d\tau
\end{aligned}
\end{equation}
for all $0<t_0<t_1<T$.\\
\indent Letting $t_1\to 0$ and then $l\to\infty$ in \eqref{eq-aligned-control-of-difference-in-L-1-at-time-t}, we have by Fatou's lemma that
\begin{equation}\label{eq-aligned-control-of-difference-in-L-1-at-time-t-after-t-1-to-0}
\begin{aligned}
&\int_{\R^n}\left[u(x,t_2)-\hat{u}(x,t_2)\right]_+\,dx\\
&\qquad \leq C\left(1-\chi_{m,q}\right)\int_{0}^{t_2}\left(\left\|v_0\right\|^{2q-m-1}_{L^{\infty}}+\left\|u(\tau)\right\|^{2q-m-1}_{L^{\infty}}+\left\|\hat{u}(\tau)\right\|^{2q-m-1}_{L^{\infty}}\right)\left\|(u-\hat{u})(\tau)\right\|_{L^1}\,d\tau\\
&\qquad \qquad \qquad +\frac{C}{\epsilon}\int_{0}^{t_2}\left(\left\|u(\tau)\right\|_{L^{\infty}}^{q-m}+\left\|\hat{u}(\tau)\right\|_{L^{\infty}}^{q-m}\right)\left\|\left(u-\hat{u}\right)(\tau)\right\|_{L^{1}}\,d\tau\\
&\qquad \qquad \qquad +C\int_{0}^{t_2}\left(\left\|u(\tau)\right\|^{q-1}_{L^{\infty}}+\left\|\hat{u}(\tau)\right\|^{q-1}_{L^{\infty}}\right)\left\|\left(u-\hat{u}\right)(\tau)\right\|_{L^{1}}\,d\tau\\
&\qquad \qquad \qquad +C\epsilon\int_{0}^{t_2}\left\|u(\tau)\right\|_{L^{\infty}}^{q-m-1}\left\|\nabla u^{m}(\tau)\right\|^2_{L^2}\,d\tau
\end{aligned}
\end{equation}
for all $0<t_1<T$.\\
\indent By symmetry, we obtain that
\begin{equation}\label{eq-aligned-control-of-difference-with-reflection-in-L-1-at-time-t-after-t-1-to-0}
\begin{aligned}
&\int_{\R^n}\left[u(x,t_2)-\hat{u}(x,t_2)\right]_+\,dx\\
&\qquad \leq C\left(1-\chi_{m,q}\right)\int_{0}^{t_2}\left(\left\|\hat{v}_0\right\|^{2q-m-1}_{L^{\infty}}+\left\|u(\tau)\right\|^{2q-m-1}_{L^{\infty}}+\left\|\hat{u}(\tau)\right\|^{2q-m-1}_{L^{\infty}}\right)\left\|(u-\hat{u})(\tau)\right\|_{L^1}\,d\tau\\
&\qquad \qquad \qquad +\frac{C}{\epsilon}\int_{0}^{t_2}\left(\left\|u(\tau)\right\|_{L^{\infty}}^{q-m}+\left\|\hat{u}(\tau)\right\|_{L^{\infty}}^{q-m}\right)\left\|\left(u-\hat{u}\right)(\tau)\right\|_{L^{1}}\,d\tau\\
&\qquad \qquad \qquad +C\int_{0}^{t_2}\left(\left\|u(\tau)\right\|^{q-1}_{L^{\infty}}+\left\|\hat{u}(\tau)\right\|^{q-1}_{L^{\infty}}\right)\left\|\left(u-\hat{u}\right)(\tau)\right\|_{L^{1}}\,d\tau\\
&\qquad \qquad \qquad +C\epsilon\int_{0}^{t_2}\left\|\hat{u}(\tau)\right\|_{L^{\infty}}^{q-m-1}\left\|\nabla \hat{u}^{m}(\tau)\right\|^2_{L^2}\,d\tau
\end{aligned}
\end{equation}
for all $0<t_0<t_1<T$. Since
\begin{equation*}
\left|u(x,t_2)-\hat{u}(x,t_2)\right|=\left[u(x,t_2)-\hat{u}(x,t_2)\right]_++\left[\hat{u}(x,t_2)-u(x,t_2)\right]_+,
\end{equation*}
by \eqref{eq-aligned-control-of-difference-in-L-1-at-time-t-after-t-1-to-0} and \eqref{eq-aligned-control-of-difference-with-reflection-in-L-1-at-time-t-after-t-1-to-0}, we establish
\begin{equation*}
\left\|\left(u-\hat{u}\right)(t_2)\right\|_{L^1}\leq C\int_{0}^{t_2}\Big(\epsilon g_1(\tau)+g_2(\tau)\left\|\left(u-\hat{u}\right)(\tau)\right\|_{L^1}\Big)\,d\tau
\end{equation*}
for all $0<t_2<T$ where $C=C(n,m,q,\gamma)$ and
\begin{equation*}
g_1(\tau)=\left\|\hat{u}(\tau)\right\|_{L^{\infty}}^{q-m-1}\left\|\nabla \hat{u}^{m}(\tau)\right\|^2_{L^2}
\end{equation*}
and
\begin{equation*}
\begin{aligned}
g_2(\tau)&=\left\|u(\tau)\right\|^{q-1}_{L^{\infty}}+\left\|\hat{u}(\tau)\right\|^{q-1}_{L^{\infty}}+\frac{1}{\epsilon}\left(\left\|u(\tau)\right\|_{L^{\infty}}^{q-m}+\left\|\hat{u}(\tau)\right\|_{L^{\infty}}^{q-m}\right)\\
&\qquad \qquad +\left(1-\chi_{m,q}\right)\left(\left\|v_0\right\|^{2q-m-1}_{L^{\infty}}+\left\|\hat{v}_0\right\|^{2q-m-1}_{L^{\infty}}+\left\|u(\tau)\right\|^{2q-m-1}_{L^{\infty}}+\left\|\hat{u}(\tau)\right\|^{2q-m-1}_{L^{\infty}}\right).
\end{aligned}
\end{equation*}
By standard O.D.E theory, we obtain that
\begin{equation*}
\left\|\left(u-\hat{u}\right)(t_2)\right\|_{L^1}\leq C\epsilon e^{\int_{0}^{t_2}g_2(\tau_1)\,d\tau_1}\int_{0}^{t_2}\frac{g_1(\tau_2)}{e^{\int_{0}^{\tau_2}g_2(\tau_3)\,d\tau_3}}\,\tau_2 \qquad \forall 0<t_2<T.
\end{equation*}
Since $0<t_2<T$ and $\epsilon>0$ are arbitrary, we conclude that
\begin{equation*}
u(x,t)=\hat{u}(x,t) \qquad \forall x\in\R^n, \,\,0\leq t<T
\end{equation*}
and the theorem follows.
\end{proof}
  {\bf Acknowledgement} 
  The authors would like to thank  Prof.   Y. Sugiyama  for suggesting   the problem considered in this  paper.  
 Ki-Ahm Lee  was supported by the National Research Foundation of Korea(NRF) grant funded by the Korea government(MSIP) (No.2014R1A2A2A01004618). 
Ki-Ahm Lee also hold a joint appointment with the Research Institute of Mathematics of Seoul National University.

\end{document}